\documentclass[leqno,11pt]{amsart}
\usepackage{amsmath,amstext,amssymb,amsopn,amsthm,mathrsfs,wasysym,dsfont,comment}

\theoremstyle{plain}

\allowdisplaybreaks

\newtheorem{thm}{Theorem}[section]

\newtheorem{ob}[thm]{Observation}
\newtheorem{conj}[thm]{Conjecture}
\newtheorem{lem}[thm]{Lemma}

\newtheorem*{ex*}{Example}

\newcommand{\N}{\mathbb{N}}

\newcommand{\R}{\mathbb{R}}
\newcommand{\Z}{\mathbb{Z}}

\def\Rr{\mathbb{R}_1^d}
\def\RR{\mathbb{R}_+^d}
		
\def\RRo{\overline{\mathbb{R}_+^d}}	
\def\1{\boldsymbol{1}}
\def\e{d'}
\def\E{\mathcal{E}}
\def\F{\mathcal{F}}
\def\W{\mathcal{W}}
\def\mE{m_\mathcal{E}}
\def\mF{m_\mathcal{F}}

\def\a{\alpha}
\def\b{\beta}
\def\ab{\alpha,\beta}
\def\t{\theta}
\def\vp{\varphi}

\DeclareMathOperator{\loc}{loc}
\DeclareMathOperator{\glob}{glob}
\DeclareMathOperator{\dom}{Dom}

\DeclareMathOperator{\spann}{span}

\def\lam{\lambda}

\title[Maximal operators of exotic and non-exotic semigroups]
	{Maximal operators of exotic and non-exotic \\ Laguerre and other semigroups
	associated with \\ classical orthogonal expansions}

\author[A. Nowak]{Adam Nowak}
\address{Adam Nowak, \newline
			Institute of Mathematics,
		Polish Academy of Sciences, \newline
      \'Sniadeckich 8,
      00--656 Warszawa, Poland    
      }
\email{anowak@impan.pl}

\author[P. Sj\"ogren]{Peter Sj\"ogren}
\address{Peter Sj\"ogren, \newline
			Mathematical Sciences, University of Gothenburg \newline
Mathematical Sciences, Chalmers University of Technology \newline
SE-412 96 G\"oteborg, Sweden 
      }
\email{peters@chalmers.se}

\author[T.Z. Szarek]{Tomasz Z. Szarek}
\address{Tomasz Z. Szarek,     \newline
			Institute of Mathematics,
		Polish Academy of Sciences, \newline
      \'Sniadeckich 8,
      00--656 Warszawa, Poland
      }
\email{szarektomaszz@gmail.com}

\begin{document}

\begin{abstract}
Classical settings of discrete and continuous orthogonal expansions, like Laguerre, Bessel and Jacobi,
are associated with second order differential operators playing the role of the Laplacian.
These depend on certain parameters of type that are usually restricted to a half-line, or a product
of half-lines if higher dimensions are considered. Following earlier research done by Hajmirzaahmad, we deal
in this paper with Laplacians in the above-mentioned contexts with no restrictions on the type parameters and bring to attention naturally
associated orthogonal systems that in fact involve the classical ones, but are different. 
This reveals new frameworks related to classical orthogonal expansions and thus new potentially rich
research areas, at least from the harmonic analysis perspective.
To support the last claim we focus on maximal operators of multi-dimensional
Laguerre, Bessel and Jacobi semigroups, with unrestricted type parameters, and prove that they satisfy
weak type $(1,1)$ estimates with respect to the appropriate measures. Generally, these measures are not
(locally) finite, which makes a contrast with the classical situations and generates new difficulties.
An important partial result of the paper is a new proof of the weak type $(1,1)$ estimate for the classical 
multi-dimensional Laguerre semigroup maximal operator.
\end{abstract}

\maketitle


\section{Introduction} \label{sec:intro}

Given a parameter $\a \in \R$, consider the Laguerre differential operator
$$
L_{\a} = - x\frac{d^2}{dx^2} - (\a+1-x)\frac{d}{dx}
$$
acting on functions on the positive half-line $\R_+ = (0,\infty)$.
There is a natural measure $\mu_{\a}$ in $\R_+$ associated with $L_{\a}$,
$$
d\mu_{\a}(x) = x^{\a}e^{-x}\, dx,
$$
which makes $L_{\a}$ formally symmetric in $L^2(d\mu_{\a})$. This is immediately seen from the factorization
$$
L_{\a}f(x) = - \big( x^{\a}e^{-x}\big)^{-1} \frac{d}{dx} \Big( x^{\a+1}e^{-x} \frac{d}{dx}f(x)\Big).
$$
Denote by $\mathcal{D}_{\a}$ the subspace of those $f \in L^2(d\mu_{\a})$ for which $L_{\a}f$ exists
in the weak sense and is in $L^2(d\mu_{\a})$, that is the distribution $L_{\a}f$ is represented by
a function that belongs to $L^2(d\mu_{\a})$.

When $\a \ge 1$, the operator $L_{\a}^{\textrm{cls}} = L_{\a}$ (here ``cls'' stands for ``classical'')
considered on the domain
$\dom L_{\a}^{\textrm{cls}} = \mathcal{D}_{\a}$ is self-adjoint. Its spectral decomposition
is given by the classical Laguerre polynomials $L_n^{\a}$, $n=0,1,2,\ldots$, which form an orthogonal
basis in $L^2(d\mu_{\a})$; one has $L_{\a} L_n^{\a} = n L_n^{\a}$. In fact, the self-adjoint operator
$L_{\a}^{\textrm{cls}}$ is characterized by
$$
L_{\a}^{\textrm{cls}}f = \sum_{n=0}^{\infty} n \, \big\langle f, \breve{L}_n^{\a}\big\rangle_{d\mu_{\a}}\,
	\breve{L}_n^{\a},
$$
and its domain
$\mathcal{D}_{\a}$ coincides with the subspace of all $f \in L^2(d\mu_{\a})$ for which this
series converges in $L^2(d\mu_{\a})$; here $\breve{L}_n^{\a} = L_n^{\a}/\|L_n^{\a}\|_{L^2(d\mu_{\a})}$
are the Laguerre polynomials normalized in $L^2(d\mu_{\a})$.

The situation is somewhat different when $-1 < \a < 1$.
The self-adjoint operator $L_{\a}^{\textrm{cls}}$ and its domain are still given by the spectral 
decomposition in terms of Laguerre polynomials, as above. But the domain is smaller than $\mathcal{D}_{\a}$, 
since a boundary condition must be imposed; more precisely 
\begin{equation} \label{domLcl}
\dom L_{\a}^{\textrm{cls}} = \big\{ f \in \mathcal{D}_{\a} : \lim_{x \to 0^+} x^{\a+1}f'(x) = 0\big\}.
\end{equation}
Actually, \eqref{domLcl} describes $\dom L_{\a}^{\textrm{cls}}$ for all $\a > -1$,
since the boundary condition is automatically satisfied for $f \in \mathcal{D}_{\a}$ in case $\a \ge 1$.
All this is well known, see \cite{H2,HK} and references given there.
Harmonic analysis related to the (self-adjoint and non-negative) `Laplacian' $L_{\a}^{\textrm{cls}}$,
$\a > -1$, 
in particular Laguerre polynomial expansions, has been extensively studied in one or more dimensions by 
various authors; see e.g.\ \cite{Di,FSS1,FSS2,GLLNU,GIT,MaSp,Mu,Mu1,Mu2,Mu3,MuWe,No,NS,Sa1,Sa2,Sa4,Sa5}.

However, much less has been done in case $\a \le -1$. The reason is that in this range of the type
parameter $\a$, the system of Laguerre polynomials is no longer contained in $L^2(d\mu_{\a})$ and,
consequently, no self-adjoint operator can be defined directly via the $L_n^{\a}$ in the same spectral
manner. Nevertheless, as discovered by Hajmirzaahmad \cite{H2}, there is another complete orthogonal
system in $L^2(d\mu_{\a})$, involving the Laguerre polynomials, and this allows one to pursue
the matters in the `exotic' case $\a \le -1$. The details are as follows.

Assume that $\a \le -1$. Then $L_{\a}^{\textrm{exo}} = L_{\a}$ has domain
$\dom L_{\a}^{\textrm{exo}} = \mathcal{D}_{\a}$ and is self-adjoint (here ``exo'' stands for ``exotic''). 
Its spectral decomposition is given in terms of the system $\{x^{-\a}L_n^{-\a} : n\ge 0\}$
which is an orthogonal basis in $L^2(d\mu_{\a})$. We have $L_{\a}(x^{-\a}L_n^{-\a}) = (n-\a)x^{-\a}L_n^{-\a}$
and
\begin{equation} \label{actLex}
L_{\a}^{\textrm{exo}}f = \sum_{n=0}^{\infty} (n-\a)\, \big\langle f,
	 x^{-\a}\breve{L}_n^{-\a}\big\rangle_{d\mu_{\a}}\, x^{-\a}\breve{L}_n^{-\a};
\end{equation}
$\dom L_{\a}^{\textrm{exo}}$ coincides with the subspace of all $f \in L^2(d\mu_{\a})$
for which the above series converges
in $L^2(d\mu_{\a})$; notice that $x^{-\a}\breve{L}_n^{-\a}$ is the normalization of
$x^{-\a}L_n^{-\a}$ in $L^2(d\mu_{\a})$.

In fact, the spectral formula \eqref{actLex} defines the self-adjoint operator $L_{\a}^{\textrm{exo}}$
as long as all $x^{-\a}L_n^{-\a}$, $n=0,1,2,\ldots$, remain in $L^2(d\mu_{\a})$, that is precisely
for all $\a < 1$. Then $L_{\a}^{\textrm{exo}}$ is identified with $L_{\a}$ understood as a 
differential operator acting on the domain
$$
\dom L_{\a}^{\textrm{exo}} = \big\{ f \in \mathcal{D}_{\a} : \lim_{x \to 0^+}
	[xf'(x)+ \a f(x)] = 0 \big\}.
$$
The boundary condition here is automatically satisfied in case $\a \le -1$.
Observe that in the overlapping range $-1< \a < 1$ the self-adjoint operators
$L_{\a}^{\textrm{cls}}$ and $L_{\a}^{\textrm{exo}}$ are different unless $\a = 0$. Furthermore, 
$L_{\a}^{\textrm{exo}}$ is non-negative
if and only if $\a \le 0$ and strictly positive when $\a < 0$.

There is also a nice probabilistic background of these considerations, see \cite[Appendix 1]{BS}.
Indeed, both $-L_{\a}^{\textrm{cls}}$, $\a > -1$, and $-L_{\a}^{\textrm{exo}}$, $\a < 0$, generate 
semigroups in $L^2(d\mu_{\a})$ that are transition probability semigroups for linear diffusions
known as the Laguerre processes. The process related to $L_{\a}^{\textrm{exo}}$ is submarkovian,
due to the nature of the left boundary point $x=0$. In particular, in the overlapping range
$-1 < \a < 0$, this boundary point is reflecting in the case of $L_{\a}^{\textrm{cls}}$ and killing
in the case of $L_{\a}^{\textrm{exo}}$. For $\a \ge 0$ the left boundary cannot be reached by the
Laguerre process and thus does not belong to the state space. A more precise description of the boundary
behavior is the following (see \cite{BS} for the terminology): $x=\infty$ is always a natural point,
but the nature of $x=0$ depends on $\a$ and it is entrance-not-exit for (classical) $\a \ge 0$, 
exit-not-entrance for (exotic) $\a \le -1$, non-singular reflecting for non-exotic $-1 < \a < 0$, and finally
non-singular killing for exotic $-1 < \a < 0$.

The principal aim of this paper is to initiate the study of the `Laplacian' $L_{\a}^{\textrm{exo}}$,
as well as its counterparts in other settings,
and the associated orthogonal expansions from a harmonic analysis perspective. This environment is
different from and more complicated than the well-studied classical case of
$L_{\a}^{\textrm{cls}}$, since for $\a \le -1$ the measure $\mu_{\a}$ is not finite near $x=0$.
More precisely, in the metric measure space
$(\R_+,\mu_{\a},|\cdot|)$ there are balls near the origin
of infinite measure and arbitrarily small radii (here $|\cdot|$ stands for the Euclidean distance).
Therefore many standard technical tools, and even intuition, fail in this context.

Our main result pertains to a general $d$-dimensional, $d \ge 1$, self-adjoint
`Laplacian' $\mathbb{L}_{\a}$
emerging from summing the action of $L_{\a_i}^{\textrm{cls}}$ or $L_{\a_i}^{\textrm{exo}}$ in each 
coordinate; now $\a \in \R^d$ is a multi-parameter.
We prove that the maximal operator of the semigroup generated by $-\mathbb{L}_{\a}$ satisfies
the weak type $(1,1)$ estimate with respect to
a measure which is the tensor product of the one-dimensional $\mu_{\a_i}$,
see Theorem \ref{thm:weak}. This implies, in particular, the almost everywhere convergence
for the semigroup to initial values taken from $L^1(d\mu_{\a})$.
We emphasize that even in the classical multi-dimensional
situation, when $\mathbb{L}_{\a}$ with $\a \in (-1,\infty)^d$ corresponds to
$\sum_{i=1}^d L_{\a_i}^{\textrm{cls}}$,
proving weak type $(1,1)$ of the maximal operator is a complicated task. The first proof
was delivered by Dinger \cite{Di}, and another one more recently by Sasso \cite{Sa4} under the restriction
$\a \in [0,\infty)^d$. 
Our methods also lead to a new complete proof of this result, simpler than the existing ones; 
 see Theorem~\ref{thm:max_Lcl}.
These new arguments are no doubt of independent interest.

It is remarkable that, in much the same spirit, exotic `Laplacians' occur in numerous other settings
well known in the literature, like those of Laguerre functions, Jacobi trigonometric polynomials and 
functions and Fourier-Bessel systems, just to mention a few.
In this paper we investigate only two further important instances related to
the Bessel and Jacobi differential operators
\begin{align*}
B_{\nu} & = -\frac{d^2}{dx^2} - \frac{2\nu+1}{x} \frac{d}{dx},\\
J_{\ab} & = -(1-x^2) \frac{d^2}{dx^2} - \big[\beta-\alpha - (\alpha+\beta+2)x\big]\frac{d}{dx},
\end{align*}
that correspond to the (modified) Hankel transform on $\R_+$ and the 
classical Jacobi polynomials on $(-1,1)$, respectively.
Harmonic analysis of self-adjoint `Laplacians' emerging from $B_{\nu}$ and $J_{\ab}$ in the classical ranges
of the parameters $\nu,\ab > -1$ has been widely investigated, see for instance
\cite{BaU,BCC0,BCC,BCDFR,BCN,BDT,BFM,BFS,BHNV,CC,CaU,CaSz,CKP,DPW,MuS,NoSj0,NoSj}
and also references therein, which is only a small part of the related literature.
In both contexts, we consider the classical and exotic self-adjoint operators
and introduce the resulting general multi-dimensional Jacobi and Bessel `Laplacians'. Then we study the 
maximal operators of the associated multi-dimensional semigroups and prove that they satisfy 
weak type $(1,1)$ estimates with respect to the appropriate measures,
see Theorems \ref{thm:weakB} and \ref{thm:weakJd}.
Again, this implies almost everywhere convergence for the semigroups to prescribed $L^1$ initial data.
It is worth pointing out that the classical and exotic Bessel and Jacobi frameworks possess probabilistic
interpretations analogous to that of Laguerre indicated above. In particular, the natures of the boundary 
points depend on the type parameters in exactly the same way.

The paper is organized as follows.
In Section \ref{sec:Lagcls} we recall the classical multi-dimensional Laguerre polynomial setting
and give a new proof of the weak type $(1,1)$ estimate for the associated Laguerre semigroup maximal 
operator.
Then we introduce in Section \ref{sec:Lagexo} a general exotic multi-dimensional Laguerre framework based
on the `Laplacian' $\mathbb{L}_{\a}$ and prove the weak type $(1,1)$ estimate for the maximal
operator of the semigroup generated by $-\mathbb{L}_{\a}$. Sections \ref{sec:Bes} and \ref{sec:Jac}
are devoted to general exotic multi-dimensional Bessel and Jacobi contexts, respectively,
and in these sections weak type $(1,1)$ estimates for the maximal operators of the general Bessel and
Jacobi semigroups are obtained.

\subsection*{Notation}
We first point out that $\RR$ always means the product of $d$ half-lines, $\RR = (0,\infty)^d$. 
Throughout the paper we use a standard notation with all symbols referring to
the metric measure spaces $(\Omega,\mu,|\cdot|)$. Here $\Omega = \RR$, $(-1,1)^d$ or $(0,\pi)^d$, and
$|\cdot|$ stands for the Euclidean norm in $\Omega$, whereas $\mu$ is a suitable measure in $\Omega$.
In particular, for $x \in \Omega$ and $r>0$ we write $B(x,r)$ to denote the open ball in $\Omega$
centered at $x$ and of radius $r$. Further, by $\langle f,g \rangle_{d\mu}$ we
mean $\int_\Omega f\overline{g}\,d\mu$ whenever the integral makes sense.

Furthermore, we use the following notation and abbreviations:
\begin{align*}
 \boldsymbol{1} &= (1,\ldots,1) \in \R^d_+, \\
 \langle \a \rangle &= \a_1+\ldots + \a_d \qquad \textrm{(length of a multi-parameter $\a \in \R^d$)},\\
 \mathds{1} &\equiv \textrm{the constant function equal to $1$ on $\Omega$}, \\
x^{\gamma} & = x_1^{\gamma_1}\cdot \ldots \cdot x^{\gamma_d}_d, \qquad x\in \R^d_+, \quad \gamma \in \R^d, \\
xy & = (x_1 y_1, \ldots, x_d y_d), \qquad x,y \in \R^d, \\
x \vee y & = \max(x,y), \qquad x,y \in \R, \\
x \wedge y & = \min(x,y), \qquad x,y \in \R.
\end{align*}

When writing estimates, we will frequently use the notation $X \lesssim Y$ to indicate that
$X \le C Y$ with a positive constant $C$ independent of significant quantities. We shall write
$X \simeq Y$ when simultaneously $X \lesssim Y$ and $Y \lesssim X$.

\section{The classical Laguerre semigroup maximal operator} \label{sec:Lagcls}

Let $d \ge 1$. Given any multi-parameter $\a \in \R^d$, we define the product measure
$\mu_{\a}$ in $\R^d_+$ by
$$
d\mu_{\a}(x) = x^{\a} e^{-(x_1+\ldots + x_d)}\, dx.
$$

The classical $d$-dimensional Laguerre polynomial setting exists for $\a \in (-1,\infty)^d$.
The system of $d$-dimensional Laguerre polynomials $L_n^{\a}$, $n \in \N^d$, which are just
tensor products $L_n^{\a} = \bigotimes_{i=1}^d L_{n_i}^{\a_i}$
of the one-dimensional Laguerre polynomials, is an orthogonal basis in $L^2(d\mu_{\a})$.
The associated differential operator is $\mathbb{L}_{\a} = \sum_{i=1}^d L_{\a_i}$ ($L_{\a_i}$
acting on the $i$th coordinate variable) and one has $\mathbb{L}_{\a} L_n^{\a} = (n_1+\ldots + n_d)L_n^{\a}$.
Actually, we consider the corresponding self-adjoint operator, denoted by the same symbol
$\mathbb{L}_{\a}$, whose spectral decomposition is given by the $L_n^{\a}$ in the canonical way.

The classical Laguerre semigroup $T_t^{\a} = \exp(-t\mathbb{L}_{\a})$, $t \ge 0$,
has in $L^2(d\mu_{\a})$ the integral representation
\begin{equation} \label{sLc}
T_t^{\a}f(x) = \int_{\R_{+}^d} G_t^{\a}(x,y) f(y)\, d\mu_{\a}(y), \qquad x \in \R^d_{+}, \quad t>0.
\end{equation}
The integral kernel here has the tensor product form
$$
G_t^{\a}(x,y) = \prod_{i=1}^d G_t^{\a_i}(x_i,y_i), \qquad x,y \in \R^d_+, \quad t > 0,
$$
where the one-dimensional kernels are given explicitly by
\begin{align} \nonumber
G_t^{\a_i}(x_i,y_i) & = 
	\sum_{n_i=0}^{\infty} e^{-t n_i} \breve{L}_{n_i}^{\a_i}(x_i) \breve{L}_{n_i}^{\a_i}(y_i) \\
& = \frac{e^{t(\alpha_i+1)/2}}{2\sinh(t/2)} 
	\exp\bigg( - \frac{e^{-t/2}}{2\sinh(t/2)}(x_i+y_i)
	\bigg) (x_i y_i)^{-\alpha_i/2} I_{\alpha_i}\bigg( \frac{\sqrt{x_i y_i}}{\sinh(t/2)}\bigg). \label{sLKer}
\end{align}
Here $x_i,y_i,t > 0$, and $I_{\nu}$ denotes the modified Bessel function of the first kind of order
$\nu> -1$, cf.\ \cite[Chapter 5]{Leb}.
From standard properties of the Bessel function, it follows that
$G_t^{\a}(x,y)$ is strictly positive and smooth in $(x,y,t) \in \R^{2d+1}_+$;
moreover, the integral in \eqref{sLc} converges absolutely for $f \in L^1(d\mu_{\a})$. 
In particular, we see that \eqref{sLc} provides a pointwise definition of $T_t^{\a}f$ for
$f \in L^1(d\mu_{\a})$, thus for all $f \in L^p(d\mu_{\a}) \subset L^1(d\mu_{\a})$, $1 \le p \le \infty$.
Further, 
observe that $T_t^{\a}\mathds{1} = \mathds{1}$, since the Laguerre polynomial
$L^{\a}_{(0,\ldots,0)}$ is constant.
Consequently, $\{T_t^{\a}\}$ is a (positive and symmetric)
semigroup of contractions on each $L^p(d\mu_{\a})$, $1 \le p \le \infty$.
In other words, $\{T_t^{\a}\}$ is a Markovian symmetric diffusion semigroup.
For all these well-known facts see e.g.\ \cite[Section~2]{NS}.

We consider the classical Laguerre semigroup maximal operator
$$
T_*^{\a}f = \sup_{t>0} |T_t^{\a}f|.
$$
By Stein's general maximal theorem for semigroups of operators \cite[p.\,73]{St}, $T^{\a}_{*}$ is bounded on
each $L^p(d\mu_{\a})$, $p > 1$. However, the case $p=1$, in which only the weak type $(1,1)$ estimate
holds, is much more subtle and cannot be dealt with by known general tools.
Nonetheless, in this section we give a new, relatively short and complete, proof of the following result.
\begin{thm} \label{thm:max_Lcl}
Let $d \ge 1$ and $\a \in (-1,\infty)^d$. Then $T^{\a}_{*}$ is bounded from $L^1(d\mu_{\a})$ to
weak $L^1(d\mu_{\a})$, that is, the estimate
$$
\mu_{\a}\big\{ x \in \R^d_+ : T_*^{\a}f(x) > \lambda \big\} \le \frac{C}{\lambda}
	\int_{\R^d_+} |f(x)|\, d\mu_{\a}(x), \qquad \lambda > 0, \quad f \in L^1(d\mu_{\a}),
$$
holds with a constant $C$ independent of $\lambda$ and $f$.
\end{thm}

In the case $d=1$, this was proved by Muckenhoupt \cite{Mu} by a rather elementary
analysis. For higher dimensions and in the diagonal case $\a = (\a_0,\ldots,\a_0)$,
Theorem \ref{thm:max_Lcl} was shown by Dinger \cite[Theorem 1]{Di}. 
Her proof is lengthy and quite technical. 
More recently Sasso \cite{Sa4}, using some ideas implemented earlier in the Hermite (Ornstein-Uhlenbeck)
context by other authors \cite{GMMST}, gave another proof of Theorem \ref{thm:max_Lcl} under the restriction
$\a\in [0,\infty)^d$. Actually, her proof is written in the one-dimensional case, but it is indicated
that the result can be extended to higher dimensions
and then the arguments needed are merely sketched (see \cite[Remarks 2.2 and 4.6]{Sa4}). 
The main tool in \cite{Sa4} is Schl\"afli's Poisson type integral representation
for the Bessel function entering the kernel, which makes the analysis of ${T}_{*}^\a$ rather long
and technical. In particular, the local balls defined in \cite{Sa4} are more complicated than those 
introduced in \cite{GMMST} and depend on an additional parameter coming from the Bessel function 
representation.

Our method of proving Theorem \ref{thm:max_Lcl} is based on the strategy presented in \cite{GMMST},
see also \cite{A}, but is considerably more
involved than in the original Hermite framework. Nevertheless, the reasoning we give seems to be
simpler and more transparent than the proofs mentioned above. 
It is perhaps worth pointing out that
the analogue of Theorem \ref{thm:max_Lcl} in the Hermite context was proved in at least four different ways.
This was done for the first time by one of the authors \cite{Sj} and then in \cite{MRS,GMMST,AFS}.

The proof of Theorem \ref{thm:max_Lcl} is contained in the subsections which follow.

\subsection{Some notation and technical preparation}
We let
\begin{align*}
S^{d-1} & = \{ x \in \R^d : |x| = 1 \}, \qquad \textrm{(unit sphere of dimension $d-1$)}, \\
\sigma & \equiv \textrm{natural (non-normalized) spherical measure on $S^{d-1}$},\\
\tilde{x} & = x/|x|, \qquad x \in \R^d \setminus \{0\}, \qquad \textrm{(projection onto $S^{d-1}$)},\\
d(\xi,w) & = \arccos \langle \xi , w \rangle, \qquad \xi,w \in S^{d-1}, \qquad 
\textrm{(geodesic distance on $S^{d-1}$)}, \\
\t(x,y) & = d(\tilde{x},\tilde{y}), \qquad x,y \in \R^d \setminus \{0\}, \qquad
	\textrm{(angle between $x$ and $y$)},\\
S^{d-1}_+ & = S^{d-1} \cap  \RRo,\\
c^+(\xi,r) & = \big\{ w \in S^{d-1}_+ : d(w,\xi) < r \big\}, 
\qquad \xi \in  S^{d-1}_+, \quad r > 0.
\end{align*}
Note that the geodesic and Euclidean distances on $S^{d-1}$ are equivalent,
\begin{align*}
\frac{2}{\pi} d(\xi,w) \le |\xi - w| \le d(\xi,w), 
\qquad \xi,w \in S^{d-1}.
\end{align*}

Next, we collect several technical lemmas needed in the sequel.
\begin{lem}\label{lem:L3}
Let $\gamma \in \mathbb{R}$ and $C>0$ be fixed. Then
\[
(ab + 1)^\gamma \exp(-C(b-a)^2) 
\lesssim
(a+1)^{2\gamma}, \qquad a,b \ge 0.
\]
\end{lem}

\begin{proof}
If $a/2 \le b \le 2a$, then $ab \simeq a^2$ and the conclusion follows.
In the opposite case, $(b-a)^2 \simeq (b+a)^2 \ge a^2 + ab$ and the asserted estimate again holds.
\end{proof}

\begin{lem}\label{lem:Gauss2}
Let $\kappa \le 0$, $\gamma \in \mathbb{R}$ and $c > 0$ be fixed and such that $\kappa + \gamma \le 0$. Then
\[
\sup_{t>0} t^\kappa (t + A)^{\gamma} \exp \left( - c \frac{z^2}{t} \right) 
\simeq
z^{2\kappa}(z^2 + A)^{\gamma},
\qquad A \ge 0, \quad z > 0.
\]
\end{lem} 

\begin{proof}
We let  $f(t)$ be the function in the supremum, and observe that it is enough to show that
$f(t)\lesssim f(z^2)$ for all $t > 0$. For  $t \ge z^2$ this follows since then
$f(t) \simeq t^\kappa (t + A)^{\gamma}$ and the last
expression is non-increasing in $t$. When $t < z^2$, we estimate the exponential factor by
const\,$\cdot (z^2/t)^\kappa$ if $\gamma \ge 0$
and by const\,$\cdot (z^2/t)^{\kappa + \gamma}$ if $\gamma < 0$.
Thus $f(t)$ is controlled by $z^{2\kappa}(t+A)^\gamma$ 
and $z^{2\kappa+2\gamma}(1+A/t)^\gamma$, respectively,
and both these expressions are non-decreasing in $t$ and 
agree with the right-hand side in the lemma for  $t = z^2$. 
\end{proof}

The following simple observation will be useful. Given $\gamma > -1$, we have
\begin{equation}\label{comp1}
\int_a^b t^\gamma \, dt \simeq (b-a) b^\gamma, \qquad b > a \ge 0.
\end{equation}

\begin{lem}\label{lem:F8}
Let $\gamma > -1$ be fixed. Then
\[
\int_a^b x^\gamma e^{-x^2} \, dx 
\simeq
\Big[ (b-a) \wedge \frac{1}{a+1} \Big] 
\big[ b \wedge (a + 1) \big]^{\gamma} e^{-a^2},
\qquad  0 \le a \le b \le \infty.
\]
\end{lem}

\begin{proof}
The case $b=\infty$ is easy, since the quotient between the two sides in question is a positive
continuous function of $a \ge 0$ having a finite and positive limit when $a \to \infty$,
by L'H\^opital's rule. Therefore we assume $b < \infty$.
It is convenient to distinguish two cases.

\noindent \textbf{Case 1:} $b \le a + \frac{1}{a+1}$.
In this situation $e^{-x^2} \simeq e^{-a^2}$, $x \in [a,b]$,
and by \eqref{comp1} the integral in question is comparable with 
$(b-a) b^{\gamma} e^{-a^2}$. Now a simple analysis leads to the required relation.

\noindent \textbf{Case 2:} $b > a + \frac{1}{a+1}$.
We have $b \ge 1$. Split the region of integration into 
$(a, a + \frac{1}{a+1})$ and $(a + \frac{1}{a+1} , b)$ denoting the corresponding integrals by $I_1$
and $I_2$, respectively. In view of Case 1, we have $I_1 \simeq (a + 1)^{\gamma - 1} e^{-a^2}$.
Therefore to finish the proof it suffices to show that 
$I_2 \lesssim (a + 1)^{\gamma - 1} e^{-a^2}$.
This, however, follows from the lemma with $b=\infty$ and the fact that
$a + \frac{1}{a+1} \simeq a + 1$.
\end{proof}

\begin{lem}\label{lem:L7L9L10}
Let $d \ge 2$ and $\gamma \in (-1,\infty)^d$ be fixed. Then
\begin{itemize}
\item[(a)] 
\[
\int_{c^+(\xi,r)} w^\gamma \, d\sigma(w)
\simeq
r^{d-1} \prod_{i=1}^d (\xi_i + r)^{\gamma_i}, 
\qquad \xi \in S^{d-1}_+, \quad 0< r \le 2\pi,
\]
\item[(b)] 
\[
\int_{c^+(\xi,r)} \Big( \prod_{i=1}^d (w_i + r)^{- \gamma_i} \Big) 
w^\gamma \, d\sigma(w)
\simeq
r^{d-1}, 
\qquad \xi \in S^{d-1}_+, \quad 0< r \le 2\pi,
\]
\item[(c)] 
\[
\int_{a \le |z| \le b}
z^\gamma e^{-|z|^2} \, dz
\simeq
\Big[ (b-a) \wedge \frac{1}{a+1} \Big] 
\big[ b \wedge (a + 1) \big]^{\langle \gamma \rangle + d -1} e^{-a^2}
\]
uniformly in $0 \le a \le b \le \infty$.
\end{itemize}
\end{lem}

\begin{proof}
Item (a) is a slight modification of \cite[(5.1.9)]{DaiXu}, the difference is that here we integrate over
subsets of $S_+^{d-1}$ and allow a wider range of $\gamma$. Since the proof is essentially a repetition of 
the arguments given in \cite[pp.\,107--109]{DaiXu}, we leave it to the reader.

Next, observe that part (c) is a straightforward consequence of Lemma~\ref{lem:F8}. Indeed, integration in 
polar coordinates produces
\begin{align*}
\int_{a \le |z| \le b}
z^\gamma e^{-|z|^2} \, dz
=
\bigg( \int_a^b t^{\langle \gamma \rangle + d -1} e^{-t^2} \, dt \bigg)
\bigg( \int_{S_+^{d-1}} w^\gamma \, d\sigma(w) \bigg),
\end{align*}
which together with Lemma~\ref{lem:F8} implies the required estimate.

Finally, we prove item (b). Let
\[
I = \{i : \xi_i > 2r\}, \qquad J = \{i : \xi_i \le 2r\}.
\]
Notice that for $w \in c^+(\xi,r)$ we have
\begin{align*}
\xi_i/2 < \xi_i - r < \xi_i - |w_i - \xi_i| \le w_i 
& \le \xi_i + |w_i - \xi_i| < \xi_i + r < 3\xi_i/2, 
\qquad i \in I, \\
w_i & \le \xi_i + |w_i - \xi_i| < \xi_i + r \le 3r, 
\qquad i \in J,
\end{align*}
which means that $w_i + r \simeq \xi_i \simeq w_i$ for $i \in I$ 
and $\xi_i + r \simeq w_i + r \simeq r$ for $i \in J$.
This leads to
\begin{align*}
\int_{c^+(\xi,r)} \Big( \prod_{i=1}^d (w_i + r)^{- \gamma_i} \Big) 
w^\gamma \, d\sigma(w)
\simeq
\Big( \prod_{i \in J} r^{- \gamma_i} \Big)
\int_{c^+(\xi,r)} \Big( \prod_{i \in J} w_i^{\gamma_i} \Big) 
\, d\sigma(w).
\end{align*}
Now an application of part (a) shows that the expression in question is comparable with 
\[
\Big( \prod_{i \in J} r^{- \gamma_i} \Big) r^{d-1} 
\Big( \prod_{i \in J} (\xi_i + r)^{\gamma_i} \Big) 
\simeq r^{d-1},
\]
and the proof of part (b) is finished.
\end{proof}

\subsection{Reformulation, reduction and the main splitting} \label{ssec:rrs}

For any $\a \in \R^d$ define the measures $\nu_{\a}$ and $\eta_{\a}$ in $\R^d_+$ by
$$
d\nu_{\a}(x) = x^{2\a+\1} e^{-|x|^2}\, dx, \qquad d\eta_{\a}(x) = x^{2\a+\1} \, dx.
$$
In the remaining part of Section \ref{sec:Lagcls} we always assume $\a \in (-1,\infty)^d$.

We now make the changes of variables $x \mapsto xx$,
$y \mapsto yy$, let $t=t(s) = 2\log\frac{1+s}{1-s}$, $s \in (0,1)$
(equivalently, $s=\tanh(t/4)$) and eliminate the Bessel function by means of the standard bounds
for $I_{\nu}$ with $\nu>-1$ (see \cite[(5.16.4) and (5.16.5)]{Leb}),
\begin{align}\label{Basy*}
I_\nu (z) \simeq z^\nu (z + 1)^{-\nu -1/2} e^z, \qquad z > 0.
\end{align}
From \eqref{sLKer} we then infer that
\begin{align*}
G^{\a}_{t(s)}\big(xx,yy\big) & \simeq s^{-d/2} \prod_{i=1}^d 
	\big[ (1-s)x_i y_i + s \big]^{-\a_i-1/2} \exp\bigg( |x|^2-\frac{|(1+s)x-(1-s)y|^2}{4s}\bigg) \\
	& =:\mathcal{G}_s^{\a}(x,y)
\end{align*}
uniformly in $x,y \in \R^d_+$ and $s\in (0,1)$. Thus the weak type $(1,1)$ of $T_*^{\a}$ with respect
to $\mu_{\a}$ (stated in Theorem \ref{thm:max_Lcl}) is equivalent to the weak type $(1,1)$ with respect to
$\nu_{\a}$ of the maximal operator
$$
f(x) \mapsto \sup_{0<s<1} \bigg| \int_{\R^d_+} \mathcal{G}_s^{\a}(x,y) f(y) \, d\nu_{\a}(y) \bigg|.
$$

Using Lemma~\ref{lem:L3} (specified to $\gamma = - \a_i -1/2$, $C=1/8$, 
$a = \frac{(1+s) x_i}{\sqrt{s}} \simeq \frac{ x_i}{\sqrt{s}}$ and
$b = \frac{(1-s) y_i}{\sqrt{s}}$, $i=1,\ldots,d$) we see that $\mathcal{G}_s^{\a}(x,y)$
is controlled by the kernel
$$
K_s^{\a}(x,y) := s^{-d/2} \prod_{i=1}^d \big(x_i+\sqrt{s}\,\big)^{-2\a_i-1}
	\exp\bigg( |x|^2-\frac{|(1+s)x-(1-s)y|^2}{8s}\bigg).
$$
This leads us to the maximal operator
$$
K^{\a}_*f(x) = \sup_{0<s<1} \int_{\R^d_+} K_s^{\a}(x,y)f(y)\, d\nu_{\a}(y), 
	\qquad x \in \R^d_+, \quad 0 \le f \in L^1(d\nu_{\a}).
$$
The following result obviously implies Theorem \ref{thm:max_Lcl}.
\begin{thm} \label{thm:max_Lclr}
Let $d \ge 1$ and $\a \in (-1,\infty)^d$. Then
$$
\nu_{\a}\big\{ x \in \R^d_+ : K^{\a}_*f(x) > \lambda \big\} 
	\le \frac{C}{\lambda} \int_{\R^d_{+}}f(x)\, d\nu_{\a}(x), 
		\qquad \lambda>0, \quad 0 \le f \in L^1(d\nu_{\a}),
$$
with a constant $C$ independent of $\lambda$ and $f$.
\end{thm}

To prove this, we follow a well-known general strategy and decompose $K^{\a}_*$ into its local and global
parts. Define
\[
m(x) = \frac{1}{|x| + 1}, \qquad x \in \RR.
\]
Our local balls will be of the type $B(x, a m(x))$, where $a > 0$ is fixed.
The crucial fact is that in such balls the measure $\nu_\a$ is proportional to the power measure
$\eta_{\a}$. More precisely, for $a > 0$ fixed, we have
\begin{align}\label{proexp}
e^{-|y|^2} \simeq e^{-|z|^2}, \qquad
d\nu_\a (y) \simeq e^{-|z|^2} d\eta_\a (y), \qquad y \in B(z, a m(z)), \quad z \in \RR.
\end{align} 
For further reference, notice that
\begin{align}\label{mest}
\frac{1}{a+1} \le \frac{m(x)}{m(x_0)} \le a + 1 
\textrm{\quad for \quad $|x - x_0| \le a$, \quad $a>0$.}
\end{align} 

The local and global parts of $K_*^\a$ are defined by
\begin{align*}
K_{*}^{\a,\loc} f(x) 
&= K_{*}^\a \big(f \chi_{B(x,m(x))} \big) (x)
= \sup_{0 < s < 1}  \int_{B(x,m(x))} K_s^{\alpha}(x,y) f(y) \, d\nu_\a(y), \\
K_{*}^{\a,\glob} f(x) 
& = K_{*}^\a \big(f \chi_{\RR \setminus B(x,m(x))} \big) (x)
= \sup_{0 < s < 1} \int_{\RR \setminus B(x,m(x))} K_s^{\alpha}(x,y) f(y) \, d\nu_\a(y).
\end{align*}
Clearly, it is enough to verify the weak type $(1,1)$ estimate for 
$K_{*}^{\a,\loc}$ and $K_{*}^{\a,\glob}$ separately. The treatment of $K_{*}^{\a,\loc}$ is relatively simple
since this operator can be controlled by means of the Hardy-Littlewood maximal function related to the
doubling measure $\eta_{\a}$. On the other hand, the analysis of $K_{*}^{\a,\glob}$
is non-standard and tricky.
As we shall see, $K_{*}^{\a,\glob}$ can be dominated by an integral operator
that turns out to be of weak type $(1,1)$.

\subsection{Treatment of the local part $K_{*}^{\a,\loc}$}
We aim at estimating the quantity
$e^{-|x|^2} K_s^{\alpha}(x,y)$ in a local ball by the so-called Gaussian bound related to the space
of homogeneous type $(\RR,\eta_\a,|\cdot|)$.
Treating the exponent in the definition of $K_s^{\alpha}(x,y)$, we get
\begin{align*}
|(1+s) x - (1-s) y|^2
& =
|(1+s) (x - y) +2sy|^2
= (1+s)^2 \Big|x - y + 2\frac{s}{1+s}y\Big|^2 \\ 
& \ge |x-y|^2 - 4s|\langle x-y, y \rangle| 
 \ge |x - y|^2 - 4s |y| m(x) \\ 
& \ge |x - y|^2 - 8s,
\end{align*}
provided that $s \in (0,1)$, $x \in \RR$ and $y \in B(x,m(x))$;
the last inequality above is a consequence of \eqref{mest} with $a=1$.
Combining this with the comparability (cf.\ \cite[Lemma 2.2]{NoSz})
\begin{align}\label{Bball}
\eta_\a \big( B(x,r) \big) \simeq
r^d \prod_{i=1}^d ( x_i + r )^{2\a_i + 1},
\qquad x \in \RR, \quad r > 0,
\end{align}
and \eqref{proexp} specified to $a=1$ and $z=x$, we obtain
\[
K_{*}^{\a,\loc} f(x) 
\lesssim
\sup_{0 < s < 1} \frac{1}{\eta_\a \big( B(x,\sqrt{s}) \big) } 
\int_{B(x,m(x))} \exp\Big( -\frac{|x - y|^2}{8s} \Big)
f(y) \, d\eta_\a(y).
\]
Since the measure $\eta_\a$ is doubling, it follows that
\begin{align}\label{est1}
K_{*}^{\a,\loc} f(x) 
\lesssim
M_\a (f \chi_{B(x,m(x))} )(x), 
\qquad x \in \RR, \quad 0 \le f \in L^1(d\nu_\a),
\end{align}
where $M_\a$ is the centered Hardy-Littlewood maximal operator associated with the space of homogeneous
type $(\RR,\eta_\a,|\cdot|)$. 
From the general theory, see \cite[Chapter 2]{H}, we know that $M_\a$ is bounded from
$L^1(d\eta_\a)$ to weak $L^{1}(d\eta_\a)$.

By a well-known covering type argument (see for example \cite[Lemma 3.2 on p.\,16]{A})
there exists a sequence of balls $B(q_k,m(q_k))$, $k=1,2,\ldots$, which cover $\RR$ and such that the 
larger concentric balls $B(q_k,3 m(q_k))$ have bounded overlap, i.e.
\[
\sum_{k=1}^\infty \chi_{B(q_k,3 m(q_k))} (y) \lesssim 1,
\qquad y \in \RR.
\]
Using \eqref{mest} with $a=1$ it is easy to check that
\[
B(x,m(x))
\subset B(q_k,3 m(q_k)) \textrm{\qquad for \quad $x \in B(q_k,m(q_k))$, \quad $k \ge 1$.}
\]
Consequently, by \eqref{est1} there exists a constant $C$ such that
\[
K_{*}^{\a,\loc} f(x) 
\le
C
M_{\a} (f\chi_{B(q_k,3 m(q_k))}) (x), 
\qquad x \in B(q_k,m(q_k)), \quad k\ge 1, \quad 0 \le f \in L^1(d\nu_\a).
\]

Now we are ready to conclude that $K_{*}^{\a,\loc}$ is of weak type $(1,1)$. Let $\lam > 0$ and
$0 \le f \in L^1(d\nu_\a)$. Applying the above estimate of $K_{*}^{\a,\loc} f(x)$
and then using \eqref{proexp} (with either $a=1$ or $a=3$, and $z=q_k$, $k \ge 1$)
and the fact that $M_\a$ is of weak type $(1,1)$ with respect to $\eta_\a$, we obtain
\begin{align*}
\nu_\a  \big\{x : K_{*}^{\a,\loc} f(x) > \lam \big\} 
& \le 
\sum_{k=1}^\infty
\nu_\a  \big\{x \in B(q_k,m(q_k)) : 
M_\a (f\chi_{B(q_k,3 m(q_k))}) (x) > \lam/C \big\}  \\
& \simeq 
\sum_{k=1}^\infty e^{-|q_k|^2}
\eta_\a  \big\{x \in B(q_k,m(q_k)) : 
M_\a (f\chi_{B(q_k,3 m(q_k))}) (x) > \lam/C \big\}  \\
& \lesssim
\sum_{k=1}^\infty e^{-|q_k|^2}
\frac{\|f \chi_{ B(q_k,3 m(q_k)) } \|_{L^1(d\eta_\a)}}{\lam}
\lesssim
\frac{\|f \|_{L^1(d\nu_\a)}}{\lam}.
\end{align*}
The conclusion follows.

\subsection{Analysis of the global part} \label{sec:Lc_glob}

We now focus on the more tricky operator $K_*^{\a,\glob}$.
To begin with, we prove a uniform estimate of $K^{\a}_s(x,y)$ outside local balls by an expression
independent of $s$, with dependence on $y$ only through the angle between $x$ and $y$, denoted $\t(x,y)$
or simply $\t$. Curiously enough, such a crude bound will be sufficient for our purpose.
\begin{lem}\label{lem:globest}
Let $d \ge 1$ and $\a \in (-1,\infty)^d$. The estimate
\begin{align}\label{globest}
K_s^{\alpha}(x,y)
\lesssim
e^{|x|^2} \bigg[ (|x|\t)^{-d} \prod_{i=1}^d (x_i + |x|\t)^{-2\a_i - 1} \bigg]
\wedge
\bigg[ (|x| + 1)^{d} \prod_{i=1}^d \Big(x_i + \frac{1}{|x|+1}\Big)^{-2\a_i - 1} \bigg]
\end{align}
holds uniformly in $s \in (0, 1)$, $x \in \RR$ and $y \in \RR \setminus B(x,m(x))$;
here $\t = \t (x,y)$ and the quantity in the first square bracket above is understood as
$\infty$ when $|x|\t = 0$.
\end{lem}

When $d=1$ the bound \eqref{globest} gives
$$
K_s^{\a}(x,y) \lesssim (x+1)^{-2\a}e^{x^2}.
$$
It is straightforward to check that the function $x\mapsto (x+1)^{-2\a}e^{x^2}$
belongs to weak $L^{1}(d\nu_{\a})$ and therefore in the one-dimensional case
$K_*^{\a,\glob}$ satisfies the weak type $(1,1)$ estimate. 

\begin{proof}[{Proof of Lemma \ref{lem:globest}}]
We first show that $K^{\a}_s(x,y)$ is controlled by $e^{|x|^2}$ times the first square bracket in 
\eqref{globest}. Since
\[
|w-z|^2 \ge |w|^2 \sin^2\t(w,z), \qquad w,z \in \R^d \setminus\{0\},
\]
we get
\begin{align}\label{est11}
|(1+s) x - (1-s) y|^2 \ge (1+s)^2 |x|^2 \sin^2\t
\simeq |x|^2 \t^2, \qquad x,y \in \RR, \quad s \in (0,1),
\end{align}
where $\t = \t(x,y) \in [0,\pi/2]$.
This estimate together with Lemma~\ref{lem:Gauss2} (applied with $\kappa = -1/2$, 
$\gamma = - \a_i -1/2$, $A=x_i^2$, $z = |x| \t$) shows that, for some $c > 0$,
the left-hand side in \eqref{globest} is controlled by
\begin{align*}
e^{|x|^2}\sup_{0 < s < 1} s^{-d/2} \prod_{i=1}^d (x_i + \sqrt{s})^{-2\a_i - 1}
\exp\Big( -c \frac{|x|^2 \t^2}{s} \Big) 
\lesssim e^{|x|^2}
(|x|\t)^{-d} \prod_{i=1}^d (x_i + |x|\t)^{-2\a_i - 1},
\end{align*}
uniformly in $x,y \in \RR$.
To finish the proof, it is enough to show that $K^{\a}_s(x,y)$ is controlled by $e^{|x|^2}$ times
the second square bracket in \eqref{globest}.

Let us first consider $s > \frac{1}{4(|x| + 1)^2}$. 
Since the function $s \mapsto s^{-1/2} (x_i + \sqrt{s})^{-2\a_i - 1}$ is decreasing,
we obtain the desired bound. In the opposite case $s \le \frac{1}{4(|x| + 1)^2} \le 1/4$ and
\begin{align*}
|(1+s) x - (1-s) y|
& = |(1-s) (x - y) +2sx|
\ge \frac{3}{4}|x-y| - 2s|x|
\ge \frac{3}{4 (|x| + 1)} - \frac{|x|}{2 (|x| + 1)^2}\\
& \ge \frac{1}{4 (|x| + 1)},
\qquad x \in \RR, \quad y \in \RR \setminus B(x,m(x)).
\end{align*}
This together with Lemma~\ref{lem:Gauss2} (specified to 
$\kappa = -1/2$, $\gamma = - \a_i -1/2$, $c = 1/(128d)$, $A=x_i^2$, 
$z = (|x| + 1)^{-1}$) produces the required estimate.
\end{proof}

Since the case $d=1$ is already done, in what follows we assume $d \ge 2$.
For $0 \le f \in L^1(d\nu_\a)$ and $x \in \RRo$ we introduce the auxiliary integral operator
\begin{align*}
& U^{\a} f(x) \\
& = e^{|x|^2} \int_{\RR} 
\bigg[ (|x|\t)^{-d} \prod_{i=1}^d (x_i + |x|\t)^{-2\a_i - 1} \bigg]
\wedge
\bigg[ (|x| + 1)^{d} \prod_{i=1}^d \Big(x_i + \frac{1}{|x|+1}\Big)^{-2\a_i - 1} \bigg] f(y) \, d\nu_\a (y);
\end{align*}
recall that $\t = \t(x,y) \in [0,\pi/2]$ and the quantity in the first square bracket above is
interpreted  as $\infty$ if $|x| \t = 0$.
It is straightforward to see that the function $\RRo \ni x \mapsto U^{\a} f(x)$ is continuous. 
Moreover, in view of Lemma~\ref{lem:globest}
\[
K_{*}^{\a,\glob} f (x)
\lesssim
U^{\a} f(x), \qquad x \in \RR, \quad 0 \le f \in L^1(d\nu_\a).
\]
So our task reduces to showing that $U^{\a}$ is of weak type $(1,1)$ with respect to $\nu_{\a}$.

Let $\lam > 0$ and $0 \le f \in L^1(d\nu_\a)$. Let $z_0 \in \RRo$ be such that
\[
|z_0| = \min \left\{|z|: z \in \overline{\RR}, \,\, 
U^{\a} f(z) \ge   \lambda  \right\}, 
\qquad r_0 := |z_0|.
\]
Such a $z_0$ exists because the level set above is closed in $\RRo$ and without any loss of generality we
may assume that it is nonempty. Observe that, in particular, we have
\begin{align}\label{est99} 
e^{r_0^2} (r_0 + 1)^d 
\prod_{i=1}^d \Big((z_0)_i + \frac{1}{r_0+1}\Big)^{-2\a_i - 1}
\|f\|_{L^1(d\nu_\a)} \ge \lam.
\end{align}
In what follows we may assume that $r_0>1$, since otherwise there exists $C > 0$ such that 
$\|f\|_{L^1(d\nu_\a)} \ge C \lam$, which forces
\[
\nu_\a  \left\{ z :  U^{\a} f(z) \ge   \lambda\right\}
\le \nu_\a (\RR) \lesssim \frac{\|f\|_{L^1(d\nu_\a)} }{ \lam }.
\]

We first verify that it is enough to consider the ring $\{ z \in \RRo : r_0 \le |z| \le 2 r_0\}$.
Using Lemma~\ref{lem:L7L9L10} (c) (with $\gamma = 2\a + \1$,
$a = 2r_0$ and $b = \infty$) and \eqref{est99}, we obtain
\begin{align*}
\nu_\a  \left\{ z \in \RRo : |z| \ge 2r_0 \right\}
& \simeq  
r_0^{2 \langle \a \rangle + 2d - 2} e^{-4r_0^2} \\
& \le 
\frac{\|f\|_{L^1(d\nu_\a)} }{ \lam }
(r_0 + 1)^d \prod_{i=1}^d \Big((z_0)_i + \frac{1}{r_0+1}\Big)^{-2\a_i - 1}
r_0^{2 \langle \a \rangle + 2d - 2} e^{-3r_0^2}.
\end{align*}
Since $(z_0)_i \le |z_0| = r_0$, we have
\[
\frac{1}{r_0+1} \le 
(z_0)_i + \frac{1}{r_0+1}
\le r_0+1, \qquad i = 1,\ldots,d,
\]
and consequently
\[
\nu_\a \left\{ z \in \RRo : |z| \ge 2r_0 \right\}
\lesssim
\frac{\|f\|_{L^1(d\nu_\a)} }{ \lam }.
\]

Thus we need only consider the region $r_0 \le |z| \le 2 r_0$. Let
\[
H = \left\{ w \in S_+^{d-1} : \textrm{there exists $\rho \in [r_0, 2r_0]$ such that
$U^{\a} f(\rho w) \ge \lam$} \right\},
\]
and for every $w \in H$ let
\[
r(w) = \min \left\{ \rho \in [r_0, 2r_0] : U^{\a} f(\rho w) \ge \lam \right\}.
\]
This definition is correct, in view of the continuity of $U^{\a} f$.
For every $w \in H$ we have $U^{\a} f(r(w) w) \ge \lam$, which means that
\begin{align*}
\lam & \le 
e^{r(w)^2} \int_{\RR} 
\bigg[ (r(w)\t)^{-d} \prod_{i=1}^d (r(w) w_i + r(w)\t)^{-2\a_i - 1} \bigg] \\
& \qquad \qquad \qquad \qquad 
\wedge
\bigg[ (r(w) + 1)^{d} \prod_{i=1}^d \Big(r(w) w_i + \frac{1}{r(w)+1}\Big)^{-2\a_i - 1} \bigg]
	f(y) \, d\nu_\a (y) \\
& \simeq
e^{r(w)^2} r_0^{-2 \langle \a \rangle - 2d}
\int_{\RR} 
\bigg[ \t^{-d} \prod_{i=1}^d (w_i + \t)^{-2\a_i - 1} \bigg] 
\wedge
\bigg[ r_0^{2d} \prod_{i=1}^d \Big(w_i + \frac{1}{r_0^2}\Big)^{-2\a_i - 1} \bigg] f(y) \, d\nu_\a (y),
\end{align*}
where $\t = \t(w,y)$.

Now an application of Lemma~\ref{lem:F8} (taken with 
$\gamma = 2 \langle \a \rangle + 2 d - 1$, $a = r(w) \simeq r_0$ and $b=\infty$) combined with the above estimate gives
\begin{align*}
& \nu_\a  \left\{ z \in \RRo : r_0 \le |z| \le 2r_0, \, \, U^{\a} f (z) \ge \lam \right\} \\
& \le 
\nu_\a  \left\{ rw : w \in H, \, \, r \ge r(w) \right\}  
=
\int_H \int_{r(w)}^\infty r^{2 \langle \a \rangle + 2d - 1} e^{-r^2} 
\, dr \, w^{2\a + \1} \, d\sigma(w) \\
& \simeq 
\int_H r_0^{2 \langle \a \rangle + 2d - 2} e^{-r(w)^2} w^{2\a + \1} \, d\sigma(w) \\
& \lesssim
\lam^{-1} r_0^{-2} \int_H \int_{\RR}
\bigg[ \t^{-d} \prod_{i=1}^d (w_i + \t)^{-2\a_i - 1} \bigg] 
\wedge
\bigg[ r_0^{2d} \prod_{i=1}^d \Big(w_i + \frac{1}{r_0^2}\Big)^{-2\a_i - 1} \bigg] \\
& \qquad \qquad \qquad \qquad \qquad \times
f(y) \, d\nu_\a (y) \, w^{2\a + \1} \, d\sigma(w).
\end{align*}
Therefore, in order to finish the proof of weak type $(1,1)$ for $U^{\a}$, it is enough to check that
\[
\int_{S_+^{d-1}}
\bigg[ \t^{-d} \prod_{i=1}^d (w_i + \t)^{-2\a_i - 1} \bigg] 
\wedge
\bigg[ r_0^{2d} \prod_{i=1}^d \Big(w_i + \frac{1}{r_0^2}\Big)^{-2\a_i - 1} \bigg]
w^{2\a + \1} \, d\sigma(w)
\lesssim
r_0^2,
\]
uniformly in $y \in \RR$ and $r_0 > 1$; here $\t = \t(w,y) = d(w, \tilde{y})$.

Let $I$ denote the last integral. We split the region of integration in $I$ into dyadic pieces,
\begin{align*}
I =
\int_{ d(w,\tilde{y}) < \frac{1}{r_0^2} } \ldots 
+ \sum_{k=1}^\infty 
\int_{ \frac{2^{k-1}}{r_0^2} \le d(w,\tilde{y}) < \frac{2^k}{r_0^2} }
\ldots,
\end{align*}
where the integration is over subsets of $S_+^{d-1}$.
Applying the second estimate in the minimum to the first term and the first one to the remaining terms,
we get
\begin{align*}
I 
&\lesssim
\int_{ d(w,\tilde{y}) < \frac{1}{r_0^2} } 
r_0^{2d} \prod_{i=1}^d \Big(w_i + \frac{1}{r_0^2}\Big)^{-2\a_i - 1} 
w^{2\a + \1} \, d\sigma(w) \\
& \quad + \sum_{k=1}^\infty 
\int_{ \frac{2^{k-1}}{r_0^2} \le d(w,\tilde{y}) < \frac{2^k}{r_0^2} }
\Big( \frac{r_0^2}{2^k} \Big)^d
\prod_{i=1}^d (w_i + d(w,\tilde{y}))^{-2\a_i - 1} 
w^{2\a + \1} \, d\sigma(w) \\
&\lesssim
\sum_{k=0}^\infty 
\Big( \frac{r_0^2}{2^k} \Big)^d
\int_{ d(w,\tilde{y}) < \frac{2^k}{r_0^2} \wedge 2\pi }
\prod_{i=1}^d \Big(w_i + \frac{2^k}{r_0^2} \wedge 2\pi \Big)^{-2\a_i - 1} 
w^{2\a + \1} \, d\sigma(w).
\end{align*}
This, however, by Lemma~\ref{lem:L7L9L10} (b) (taken with 
$\gamma = 2\a + \1$, $\xi = \tilde{y}$ and $r = \frac{2^k}{r_0^2} \wedge 2\pi$) leads to
\[
I \lesssim
\sum_{k=0}^\infty 
\Big( \frac{r_0^2}{2^k} \Big)^d
\Big( \frac{2^k}{r_0^2} \wedge 2\pi  \Big)^{d-1}
\le 
r_0^2 \sum_{k=0}^\infty 2^{-k}
\simeq r_0^2,
\]
which is the desired estimate. 

The weak type $(1,1)$ estimate for $U^{\a}$ is proved, and it follows that
$K_{*}^{\a,\glob}$ is of weak type $(1,1)$ with respect to $\nu_\a$.
This finishes proving Theorem \ref{thm:max_Lclr}, thus also Theorem \ref{thm:max_Lcl}.

\section{Exotic Laguerre semigroup maximal operator} \label{sec:Lagexo}

To begin with, we consider the one-dimensional situation. 
\subsection{Description of the exotic Laguerre context in dimension one}
Let $\a \in \R$.
Recall that the classical (non-exotic) one-dimensional Laguerre setting exists for $\a > -1$
and the corresponding
self-adjoint Laplacian is $L_{\a}^{\textrm{cls}}$. The semigroup generated by $L_{\a}^{\textrm{cls}}$
has the integral representation with the explicit integral kernel, see \eqref{sLc} and \eqref{sLKer},
$$
G_t^{\a}(x,y) = \frac{e^{t(\alpha+1)/2}}{2\sinh(t/2)} 
	\exp\bigg( - \frac{e^{-t/2}}{2\sinh(t/2)}(x+y)
	\bigg) (x y)^{-\alpha/2} I_{\alpha}\bigg( \frac{\sqrt{x y}}{\sinh(t/2)}\bigg).
$$

Now, the exotic situation occurs for $0 \neq \a < 1$ when the associated self-adjoint Laplacian
is $L_{\a}^{\textrm{exo}}$ (recall that $L_{0}^{\textrm{exo}} = L_0^{\textrm{cls}}$). 
The orthogonal basis of $L^2(d\mu_{\a})$ underlying the spectral decomposition
of $L_{\a}^{\textrm{exo}}$ is $\{x^{-\a}L_n^{-\a} : n \ge 0\}$ and we have
$L_{\a}^{\textrm{exo}}(x^{-\a}L_n^{-\a}) = (n-\a)x^{-\a}L_n^{-\a}$, $n=0,1,2,\ldots$.
Therefore the exotic Laguerre semigroup $\widetilde{T}_t^{\a} = \exp(-tL_{\a}^{\textrm{exo}})$, $t \ge 0$, is
in $L^2(d\mu_{\a})$ given by
$$
\widetilde{T}_t^{\a}f = \sum_{n=0}^{\infty} e^{-t(n-\a)} 
	\big\langle f, x^{-\a}\breve{L}_n^{-\a}\big\rangle_{d\mu_{\a}} \, x^{-\a}\breve{L}_n^{-\a}, \qquad t \ge 0.
$$
The corresponding integral representation is
\begin{equation} \label{eLc}
\widetilde{T}_t^{\a}f(x) = \int_0^{\infty} \widetilde{G}_t^{\a}(x,y) f(y)\, d\mu_{\a}(y),
	\qquad x, t >0,
\end{equation}
with the integral kernel
\begin{align} \nonumber
\widetilde{G}_t^{\a}(x,y) &= \sum_{n=0}^{\infty} e^{-t(n-\a)} \,
	x^{-\a}\breve{L}_n^{-\a}(x)\, y^{-\a}\breve{L}_n^{-\a}(y) \\
&	= e^{t\a} (xy)^{-\a} G_t^{-\a}(x,y), \qquad x,y,t>0. \label{ecl}
\end{align} 
The fact that the exotic kernel is expressed directly in terms of the classical one is crucial for our
developments. In particular, we see that $\widetilde{G}_t^{\a}(x,y)$ is strictly positive and smooth
in $(x,y,t) \in \R_+^3$; this, of course, follows from analogous properties of the non-exotic kernel.

Using the standard bounds for the Bessel function \eqref{Basy*},
it is straightforward to verify that in case $\a < 0$ the integral in \eqref{eLc} converges absolutely
for $f \in L^p(d\mu_{\a})$, $1 \le p \le \infty$.
Thus \eqref{eLc} provides a pointwise definition of $\widetilde{T}_t^{\a}$, $t>0$, on the $L^p$ spaces.
Moreover, $\{\widetilde{T}_t^{\a}\}$ is a semigroup of operators on each 
$L^p(d\mu_{\a})$, $1 \le p \le \infty$; this can be checked, for instance, by means of the 
relation \eqref{ecl} and the analogous property of $\{T_t^{-\a}\}$.

The situation is more subtle in case $0 < \a < 1$. Then the Bessel function asymptotics reveal that
\eqref{eLc} provides a definition of $\widetilde{T}_t^{\a}$ on $L^p(d\mu_{\a})$ only if $p > \a + 1$.
Indeed, if $p \le \a + 1$ then there is an $f \in L^p(d\mu_{\a})$ such that the integral
in \eqref{eLc} diverges for all $x,t>0$. Moreover, if we assume $p > \a+1$ and, given $t>0$,
require that $\widetilde{T}_t^{\a}f \in L^p(d\mu_{\a})$ for all $f \in L^p(d\mu_{\a})$, then we arrive at the
dual restriction $p < (\a+1)/\a$. Thus, in case $0 < \a < 1$, the exotic Laguerre
semigroup maps $L^p$ into itself 
only for $\a+1 < p < (\a+1)/\a$. This is an instance of the
so-called \emph{pencil phenomenon} occurring also in other Laguerre frameworks, cf.\ \cite{MST,NoSj1}.

For further reference we note that in the overlapping range $-1< \a < 0$ the exotic kernel
is dominated by the classical one,
\begin{equation} \label{den}
	\widetilde{G}_t^{\a}(x,y) < G_t^{\a}(x,y), \qquad x,y,t>0, \quad -1 < \a < 0.
\end{equation}
This is an easy consequence of the following inequality satisfied by $I_{\nu}$, see \cite[Theorem 1]{Lo},
\begin{equation} \label{iI}
I_{\nu+\varepsilon}(z) < I_{\nu}(z), \qquad z > 0, \quad \nu \ge -\varepsilon/2, \quad \nu > -1, \quad
	\varepsilon > 0.
\end{equation}
We remark that
\eqref{den} is quite obvious, at least heuristically, in view of the probabilistic interpretation,
see Section \ref{sec:intro}. Indeed, the two kernels are transition probability densities
for processes that are distinguished only by the nature of the boundary point $x=0$,
which is killing or reflecting, respectively. Roughly speaking, one of these processes is just
the other one killed upon hitting the boundary. 

Another fact we shall need is that for $\a < 0$ and $t >0$, the operator $\widetilde{T}_t^{\a}$ is
contractive on $L^{\infty}$. More precisely, we have even strict inequality in the estimate
\begin{equation} \label{cnex}
\widetilde{T}_t^{\a} \mathds{1}(x) < 1, \qquad x,t>0, \quad \a < 0.
\end{equation}
This is justified as follows. By the explicit formula for $\widetilde{G}_t^{\a}(x,y)$, see \eqref{ecl},
we have
$$
\widetilde{T}_t^{\a}\mathds{1}(x) = \frac{e^{t(\a+1)/2}}{\sinh(t/2)} \, x^{-\a/2} 
	\exp\bigg(-\frac{e^{-t/2}}{2\sinh(t/2)}x\bigg) \; \mathcal{I}(x),
$$
where 
$$
\mathcal{I}(x) = \int_0^{\infty} y^{\a+1} \exp\bigg( -\frac{e^{t/2}}{2\sinh(t/2)} y^2\bigg)
	I_{-\a}\bigg( \frac{\sqrt{x}}{\sinh(t/2)}y\bigg)\, dy.
$$
The integral here can be computed by means of \cite[Lemma 2.2]{NS}. The result is
$$
\mathcal{I}(x) = \frac{2^{\a}}{\Gamma(1-\a)} \big[\sinh(t/2)\big]^{\a+1} e^{-t/2} x^{-\a/2}\,
	{_1F_1}\bigg( 1; 1-\a; \frac{e^{-t/2}}{2\sinh(t/2)}x\bigg),
$$
${_1F_1}$ denoting Kummer's confluent hypergeometric function. This leads to the formula
$$
\widetilde{T}_t^{\a}\mathds{1}(x) = H_{1,1-\a}\bigg( \frac{e^{-t/2}x}{2\sinh(t/2)}\bigg), \qquad x,t > 0,
$$
where $H_{1,1-\a}$ is the function defined in \cite[Section 2]{NS}. 
Now we can easily conclude \eqref{cnex} by the proof of \cite[Lemma 2.3]{NS}. Furthermore,
\eqref{cnex} cannot be improved (i.e.\ the right-hand side cannot be smaller) since, 
in view of \cite[Lemma 2.3 (a)]{NS}, $H_{1,1-\a}(u) \to 1$ as $u \to \infty$.

\subsection{Multi-dimensional exotic Laguerre context and the maximal theorem}
We now pass to the multi-dimensional situation that arises, roughly speaking, by taking a tensor product
of the one-dimensional classical and exotic Laguerre settings. 
Let $d \ge 1$. We associate with each $\E \subset \{1,\ldots,d\}$ a set of multi-parameters
\begin{align}\label{def:ABC}
A(\E) & = \{\a \in \R^d : 0\neq \a_i < 1 \;\;\textrm{for}\;\; i \in \E \;\; \textrm{and} \;\;
\a_i > -1 \;\;\textrm{for}\;\; i \in \E^c\};
\end{align}
here and elsewhere $\E^c$ stands for the complement of $\E$ in $\{1,\ldots,d\}$.
The set $\E$ will indicate which coordinate axes are exotic.
From now on we assume that $\E$ is fixed, and we always consider $\a \in A(\E)$.
Further, for such $\a$ we let 
\[
\mE(\a) = 
\begin{cases}	
	\max\{\a_i : i \in \E\}, & \quad \E \ne \emptyset, \\
	- \infty, & \quad \E = \emptyset.
\end{cases}
\]

Define
$$
\mathfrak{L}_n^{\a,\E} = \bigotimes_{i=1}^d 
\begin{cases}
	x_i^{-\a_i} \breve{L}_{n_i}^{-\a_i}, & \quad i \in \E,\\
	\breve{L}_{n_i}^{\a_i}, & \quad i \notin \E,
\end{cases}	
	\qquad n \in \N^d.
$$
Then the system $\{\mathfrak{L}_{n}^{\a,\E} : n \in \N^d\}$ is an orthonormal basis in $L^2(d\mu_{\a})$.
These are eigenfunctions of the Laguerre differential operator $\mathbb{L}_{\a} = \sum_{i=1}^d L_{\a_i}$
(recall that here $L_{\a_i}$ acts on the $i$th coordinate variable), we have
$\mathbb{L}_{\a} \mathfrak{L}_{n}^{\a,\E} = \lambda_n^{\a,\E} \mathfrak{L}_n^{\a,\E}$, 
$n \in \N^d$, where the eigenvalues are given by
$$
\lambda_n^{\a,\E} =  \sum_{i\in \E} (n_i-\a_i) + \sum_{i \in \E^c} n_i
	= \sum_{i = 1}^d n_i - \sum_{i \in \E} \a_i, \qquad n \in \N^d.
$$
Here and later on we use the standard conventions concerning empty sums and products.

We consider the self-adjoint extension of $\mathbb{L}_{\a}$, acting initially on
$\spann\{\mathfrak{L}_n^{\a,\E}: n\in \N^d\} \subset L^2(d\mu_{\a})$, defined by
$$
\mathbb{L}_{\a,\E} f = \sum_{n \in \N^d} \lambda_n^{\a,\E} \big\langle f, \mathfrak{L}_n^{\a,\E}
	\big\rangle_{d\mu_{\a}} \, \mathfrak{L}_{n}^{\a,\E}
$$
on the domain $\dom \mathbb{L}_{\a,\E}$ consisting of all those $f \in L^2(d\mu_{\a})$ for which
this series converges in $L^2(d\mu_{\a})$. Notice that $\mathbb{L}_{\a,\E}$ is non-negative
in the spectral sense if and only if $\sum_{i \in \E} \a_i \le 0$. Observe also that with
$\E = \emptyset$ we recover the classical multi-dimensional Laguerre polynomial context considered
in Section \ref{sec:Lagcls}. Otherwise, i.e.\ when $\E \neq \emptyset$, we use the adjective
\emph{exotic} to distinguish this situation and related objects from the classical setup.

The semigroup $\mathbb{T}_t^{\a,\E} = \exp(-t\mathbb{L}_{\a,\E})$, $t \ge 0$,
defined spectrally in $L^2(d\mu_{\a})$, has the
integral representation
\begin{equation} \label{eLir}
\mathbb{T}_t^{\a,\E}f(x) = \int_{\R^d_+} \mathbb{G}_t^{\a,\E}(x,y)f(y)\, d\mu_{\a}(y),
	\qquad x \in \R^d_+, \quad t > 0,
\end{equation}
where the integral kernel is the product of one-dimensional classical and exotic kernels,
$$
\mathbb{G}_t^{\a,\E}(x,y) = 
	\prod_{i \in \E} \widetilde{G}_t^{\a_i}(x_i,y_i)
	\prod_{i \in \E^c} G_t^{\a_i}(x_i,y_i), \qquad  x,y \in \R^d_+, \quad t>0.
$$
Clearly, $\mathbb{G}_t^{\a,\E}(x,y)$ is strictly positive and symmetric, and moreover smooth
in $(x,y,t) \in \R^{2d+1}_+$; this, as well as several further facts below,
follows from analogous properties of the one-dimensional kernels.

When $\mE(\a) < 0$, the formula \eqref{eLir} provides a pointwise definition of
$\mathbb{T}_t^{\a,\E}f$ on all $L^p(d\mu_{\a})$ spaces, $1\le p \le \infty$.
Moreover, $\{\mathbb{T}_t^{\a,\E}\}$
is a semigroup of contractions on each $L^p(d\mu_{\a})$, $1\le p \le \infty$. This follows
(see \cite[Lemma 2.1]{NS}) from the
estimate $\mathbb{T}_t^{\a,\E} \mathds{1}(x) \le 1$, $x \in \R^d_{+}$, $t>0$, where the inequality
is actually strict provided that $\E \neq \emptyset$. Thus $\{\mathbb{T}_t^{\a,\E}\}$ is
a submarkovian symmetric diffusion semigroup, which is Markovian if and only if $\E = \emptyset$.

For $\a$ satisfying $\mE(\a) > 0$ a pencil phenomenon occurs. The operators
$\mathbb{T}_t^{\a,\E}$, $t >0$, are not even defined in $L^p(d\mu_{\a})$ when
$p \le 1 + \mE(\a)$; in particular, this happens for $p=1$. On the other hand,
they are well defined for $p > 1 + \mE(\a)$, but then the requirement that, given $t>0$,
$\mathbb{T}_t^{\a,\E}f \in L^p(d\mu_{\a})$ for all $f \in L^p(d\mu_{\a})$ forces the dual restriction
$p < 1 + 1/\mE(\a)$; in particular, $\mathbb{T}_t^{\a,\E}\mathds{1}$
is an unbounded function for each $t>0$.

Bring now in the maximal operator
$$
\mathbb{T}_{*}^{\a,\E}f = \sup_{t>0} \big| \mathbb{T}_t^{\a,\E}f\big|.
$$
Our aim is to prove that $\mathbb{T}_*^{\a,\E}$ satisfies the weak type $(1,1)$ estimate with respect
to the measure $\mu_{\a}$. In view of what was already said, this problem makes sense only when
$\mE(\a) < 0$. Note that for such $\a$, Stein's general maximal theorem for semigroups of operators
\cite[p.\,73]{St} implies the $L^p(d\mu_{\a})$-boundedness of $\mathbb{T}_{*}^{\a,\E}$ for $p>1$.
Because of the pencil phenomenon, from the perspective of the $L^p$ mapping properties of
$\mathbb{T}_*^{\a,\E}$ the case $\mE(\a) > 0$ is qualitatively different and
more sophisticated than $\mE(\a) < 0$; thus it is beyond the scope of this paper.
We refer to \cite{NoSj1} for some interesting questions
that can be posed in connection with the pencil phenomenon.

The main result of this paper reads as follows.
\begin{thm} \label{thm:weak}
Let $d \ge 1$ and $\a \in A(\E)$ for some $\E \subset \{1,\ldots,d\}$. Assume that $\mE(\a) < 0$. 
Then $\mathbb{T}_{*}^{\a,\E}$ is bounded from $L^1(d\mu_{\a})$ to weak $L^{1}(d\mu_{\a})$,
that is, the estimate
$$
\mu_{\a}\big\{x \in \R^d_{+}: \mathbb{T}_{*}^{\a,\E}f(x)> \lambda \big\} \le \frac{C}{\lambda}
	\int_{\R^d_{+}} |f(x)|\, d\mu_{\a}(x), \qquad \lambda > 0, \quad f \in L^1(d\mu_{\a}),
$$
holds with a constant $C$ independent of $\lambda$ and $f$.
\end{thm}

When proving Theorem \ref{thm:weak} we can make the following reductions.
\begin{itemize}
\item[(R1)] Assume $f \ge 0$, since the kernel $\mathbb{G}_t^{\a,\E}(x,y)$ is positive.
\item[(R2)] Restrict to $\a$ satisfying $\mE(\a) \le -1$, because of the majorization \eqref{den}. 
\item[(R3)] Consider $\E \neq \emptyset$, since the case $\E =\emptyset$ 
	corresponds to the classical case in which the result is already known, see Theorem \ref{thm:max_Lcl}.
\item[(R4)] Drop $\E$ from the notation, in view of (R2).
\item[(R5)] Assume, for symmetry reasons, that $\E = \{1, \ldots, \e \}$ for some $1 \le \e \le d$.
\end{itemize}
Then $\a = (\a',\a'') \in (-\infty,-1]^{\e} \times (-1,\infty)^{d-\e}$ and the kernel can be written as
$$
\mathbb{G}_t^{\a}(x,y) = \prod_{i=1}^{\e} \widetilde{G}_t^{\a_i}(x_i,y_i) 
	\prod_{i=\e+1}^d G_t^{\a_i}(x_i,y_i)
	\equiv \widetilde{G}_t^{\a'}(x',y') G_t^{\a''}(x'',y''),
$$
where for $z \in \R^d$ we denote $z'=(z_1,\ldots,z_{\e})\in \R^{\e}$ and 
$z''=(z_{\e+1},\ldots,z_d)\in \R^{d-\e}$.
Note that the double-prime part may be void here and in what follows.

Taking into account \eqref{ecl} and the considerations and notation from Section \ref{ssec:rrs}, we see that
$$
\mathbb{G}_{t(s)}^{\a}\big( xx,yy \big) \lesssim \mathbb{K}_s^{\a}(x,y),
	\qquad x,y \in \R^d_+, \quad s \in (0,1),
$$
where
$$
\mathbb{K}_{s}^{\a}(x,y) := 
	(1-s)^{-2 \langle \a' \rangle}(x'y')^{-2\a'} K_s^{-\a'}(x',y') K_s^{\a''}(x'',y'').
$$
This leads us to considering the maximal operator
$$
\mathbb{K}^{\a}_{*}f(x)  = \sup_{0<s<1} \int_{\R^d_+} \mathbb{K}_s^{\a}(x,y) f(y)\, d\nu_{\a}(y),
	\qquad x \in \R^d_+, \quad 0 \le f \in L^1(d\nu_{\a}).
$$
Clearly, Theorem \ref{thm:weak} will follow once we prove the result for $\mathbb{K}^{\a}_{*}$
stated below.
\begin{thm} \label{thm:weakr}
Let $d \ge 1$ and $\a \in (-\infty,-1]^{\e} \times (-1,\infty)^{d-\e}$ for some $1\le \e \le d$.
Then $\mathbb{K}_{*}^{\a}$ satisfies
$$
\nu_{\a}\big\{x\in \R^d_{+}: \mathbb{K}_*^{\a}f(x)> \lambda \big\} \le \frac{C}{\lambda}
	\int_{\R_+^d}f(x)\, d\nu_{\a}(x), \qquad \lambda > 0, \quad 0 \le f \in L^1(d\nu_{\a}),
$$
with a constant $C$ independent of $\lambda$ and $f$.
\end{thm}

We split the proof into several lemmas stated below and proved in the subsequent subsections.
Altogether, they imply Theorem \ref{thm:weakr}, taking
into account the product structure of $\mathbb{K}_s^{\a}(x,y)$ and Theorem \ref{thm:max_Lclr}.

\begin{lem} \label{lem:p_loc}
For any $d \ge 1$, $1 \le \e \le d$ and each $\a \in (-\infty,-1]^{\e} \times (-1,\infty)^{d-\e}$
the maximal operator
$$
\mathbb{K}^{\a}_{*,1}f(x) = \sup_{0< s < 1} \int_{\R^d_+} 
	\chi_{\{x_i/2 \le y_i \le 2x_i\; \textrm{for}\; i=1,\ldots,\e\}} \mathbb{K}^{\a}_s(x,y)
	 f(y)\, d\nu_{\a}(y), \qquad f \ge 0,
$$
is of weak type $(1,1)$ with respect to $\nu_{\a}$.
\end{lem}

\begin{lem} \label{lem:p_glob_0}
For each $\beta \le -1$ the one-dimensional operator
$$
N^{\beta}_{1}f(x) = \int_{0}^{\infty} \chi_{\{y < x/2 \;\textrm{or}\; y>2x\}} 
	\bigg[(xy)^{-2\beta} \sup_{0<s \le 1/4} 
	(1-s)^{-2\beta}  K_s^{-\beta}(x,y)\bigg] f(y)\, d\nu_{\beta}(y), \qquad f \ge 0,
$$
is of strong type $(1,1)$ with respect to $\nu_{\beta}$.
\end{lem}

\begin{lem} \label{lem:p_lg1}
For each $\beta \le -1$ and $\gamma> -1$ the one-dimensional operators
\begin{align*}
N^{\beta}_{2}f(x) & = \chi_{(0,1)}(x)\int_{0}^{\infty}
	\bigg[ (xy)^{-2\beta} \sup_{1/4 < s < 1}
	(1-s)^{-2\beta} K_s^{-\beta}(x,y)\bigg] f(y)\, d\nu_{\beta}(y), \qquad f \ge 0, \\
N^{\gamma}_{3}f(x) & = \chi_{(0,1)}(x)\int_{0}^{\infty}
\bigg[\sup_{1/4 < s < 1}	K_s^{\gamma}(x,y)\bigg] f(y)\, d\nu_{\gamma}(y), \qquad f \ge 0,
\end{align*}
are of strong type $(1,1)$ with respect to $\nu_{\beta}$ and $\nu_{\gamma}$, respectively.
\end{lem}
One could strengthen Lemma \ref{lem:p_lg1} by moving the characteristic functions under the integrals
and then replacing them by $\chi_{\{x/y \ge 3/4 \;\textrm{or}\; x \le 1\}}$. Since this is not needed for
our purpose, we leave the details to interested readers.

\begin{lem} \label{lem:p_gg1}
For any $d \ge 1$, $1 \le \e \le d$ and each $\a \in (-\infty,-1]^{\e} \times (-1,\infty)^{d-\e}$
the operator
$$
\mathbb{K}^{\a}_{*,2}f(x) = 	\chi_{\Rr}(x) \int_{\R^d_{+}} 
 \sup_{1/4<s<1} \mathbb{K}_s^{\a}(x,y) f(y)\, d\nu_{\a}(y),
	\qquad f \ge 0,
$$
is of weak type $(1,1)$ with respect to $\nu_{\a}$; here $\Rr := [1,\infty)^d$.
\end{lem}

It remains to give proofs of the four lemmas.

\subsection{Proofs of Lemmas \ref{lem:p_loc}-\ref{lem:p_lg1}}

\begin{proof}[{Proof of Lemma \ref{lem:p_loc}}]
What we need to prove turns out to be a consequence of the classical (non-exotic) result stated in
Theorem \ref{thm:max_Lclr}.
The relation $x' \sim y'$ means that $x_i/2 \le y_i \le 2x_i$ for $i=1,\ldots,\e$.
Observe that
\begin{align*}
\mathbb{K}_{*,1}^{\a}f(x)
& \le \sup_{0< s < 1} \int_{y'\sim x'} 
(x' y')^{-2\a'}
K_s^{-\a'}(x',y')K_s^{\a''}(x'',y'') f(y)\, 
	d\nu_{\a}(y)\\
& \simeq \sup_{0< s < 1} \int_{y'\sim x'} K_s^{\check{\a}}(x,y) f(y)\, d\nu_{\check{\a}}(y),
\end{align*}
where $\check{\a} = (-\a',\a'') \in (-1,\infty)^d$. Thus, denoting for $n \in \Z^{\e}$
\begin{align*}
S_n & = \big\{x' \in \R_+^{\e} : 2^{n_i} < x'_i \le 2^{n_i+1} \; 
	\textrm{for}\; i=1,\ldots,\e\big\} \times \R^{d-\e}_+,\\
{S}^{\angle}_n & = \big\{x' \in \R_+^{\e} : 2^{n_i-1} < x'_i \le 2^{n_i+2} \; 
	\textrm{for}\; i=1,\ldots,\e\big\} \times \R^{d-\e}_+,
\end{align*}
we get
\begin{align*}
\nu_{\a}\{x\in\R^d_{+} : \mathbb{K}^{\a}_{*,1}f(x) > \lambda\} 
& = \sum_{n \in \Z^{\e}} \nu_{\a}\{x\in S_n : \mathbb{K}^{\a}_{*,1}(\chi_{{S}^{\angle}_n}f)(x) > \lambda\} \\
& \lesssim \sum_{n \in \Z^{\e}} 2^{4 \langle \a' n \rangle} \nu_{\check{\a}}\{ x\in S_n : 
	K_{*}^{\check{\a}}(\chi_{{S}^{\angle}_n}f)(x) > c\lambda\},
\end{align*}
where $c>0$ depends only on $\a$. Since $K_{*}^{\check{\a}}$ is of weak type $(1,1)$ with respect
to $\nu_{\check{\a}}$, see Theorem~\ref{thm:max_Lclr}, we further obtain
\begin{align*}
\nu_{\a}\big\{x\in\R^d_{+} : \mathbb{K}^{\a}_{*,1}f(x) > \lambda\big\}
& \lesssim \sum_{n \in \Z^{\e}} 2^{4 \langle \a' n \rangle}\frac{1}{\lambda} \int_{{S}^{\angle}_n}f(x)\,
	d\nu_{\check{\a}}(x) 
\simeq \frac{1}{\lambda} \sum_{n \in \Z^{\e}} \int_{{S}^{\angle}_n}f(x)\, d\nu_{\a}(x).
\end{align*}
Since the ${S}_n^{\angle}$ have finite overlap, the last sum is comparable with $\|f\|_{L^1(d\nu_{\a})}$.
The conclusion follows.
\end{proof}

\begin{proof}[{Proof of Lemma \ref{lem:p_glob_0}}]
The kernel of $N^{\beta}_{1}$ is comparable with, see Section \ref{ssec:rrs},
$$
N_1(x,y) = \chi_{\{y < x/2 \;\textrm{or}\; y>2x\}} (xy)^{-2\beta} 
\sup_{0<s \le 1/4} s^{-1/2}
\big(x+\sqrt{s}\big)^{2\beta-1}
	\exp\bigg( - \frac{|(1+s)x-(1-s)y|^2}{8s}\bigg) e^{x^2}.
$$
By the triangle inequality we see that
\[
|(1+s)x-(1-s)y| \ge \frac{x \vee y}{8}, \qquad 0<s \le 1/4,
\]
provided that $y < x/2$ or $y>2x$. Combining this with Lemma \ref{lem:Gauss2} (specified to 
$\kappa = -1/2$, $\gamma = \b -1/2$, $c = 1/512$, $A=x^2$, $z = x \vee y$) we obtain
\begin{align*}
N_1(x,y) & \le (xy)^{-2\beta} 
\sup_{0<s \le 1/4} s^{-1/2}
\big(x+\sqrt{s}\big)^{2\beta-1}
\exp\bigg(-\frac{(x \vee y)^2}{512s}\bigg)e^{x^2} \\
& \lesssim (xy)^{-2\beta} (x \vee y)^{2\beta-2} e^{x^2}
= (x \wedge y)^{-2\beta} (x \vee y)^{-2} e^{x^2}.
\end{align*}
Therefore the proof will be finished once we ensure that
\[
\int_{0}^{\infty} (x \wedge y)^{-2\beta} (x \vee y)^{-2} e^{x^2} \, d\nu_{\beta}(x)
\lesssim
1, \qquad y > 0,
\]
which is straightforward.
\end{proof}

\begin{proof}[{Proof of Lemma \ref{lem:p_lg1}}]
First we treat $N^{\beta}_{2}$.
Taking into account that $x \le 1$, $1/4< s< 1$ and $\beta$ is negative, we see that the relevant
kernel is controlled by
$$
N_2(x,y) = \chi_{(0,1)}(x) (xy)^{-2\beta} \sup_{1/4<s<1} (1-s)^{-2\beta} 
	\exp\Big( -\frac{1}{16} |(1-s)y|^2\Big).
$$
This implies $N_2(x,y) \lesssim \chi_{(0,1)}(x) x^{-2\beta}$. Since
$\int_0^1 x^{-2\beta} d\nu_{\beta}(x) < \infty$, the conclusion follows.

Passing to $N^{\gamma}_3$, it is immediate to see that the kernel is dominated by a constant independent
of $x$ and $y$, so the conclusion is trivial ($\nu_{\gamma}$ is a finite measure).
\end{proof}

\subsection{Proof of Lemma \ref{lem:p_gg1}}

We now show the remaining, more difficult Lemma \ref{lem:p_gg1}.
We will need two auxiliary results which are counterparts of
Lemmas \ref{lem:L3}, \ref{lem:F8} and \ref{lem:L7L9L10}.

\begin{lem}\label{lem:L3_2}
Let $\gamma \le 0$ and $C>0$ be fixed. Then
\[
(ab)^{-2 \gamma} \exp(-C(b-a)^2) 
\lesssim
a^{- 4\gamma}, \qquad a \ge 1, \quad b > 0.
\]
\end{lem}

\begin{proof}
Since $(ab)^{-2 \gamma} \le (ab + 1)^{-2 \gamma}$ and $a+1 \simeq a$,
the asserted estimate is a direct consequence of Lemma \ref{lem:L3}.
\end{proof}

Notice that Lemma~\ref{lem:L3_2} is not true for all $a,b>0$ (take $a \to 0$ and $b\simeq 1$).

\begin{lem}\label{lem:F8&L7L9L10_2}
Let $\gamma \in \R$, $d \ge 1$ and $\a \in \R^d$ be fixed. Then one has the estimates
\begin{itemize}
\item[(a)]
\[
\int_a^\infty x^\gamma e^{-x^2} \, dx 
\simeq
a^{\gamma - 1} e^{-a^2},
\qquad  a \ge 1, 
\]
\item[(b)]
\[
\nu_\a  \big\{z \in \Rr : |z| \ge a \big\}
\lesssim
e^{-3a^2/4}, \qquad  a \ge 1. 
\]
\end{itemize}
\end{lem}

\begin{proof}
Part (a) is trivial, so let us focus on item (b).
We assume that $d \ge 2$ since the case $d=1$ easily follows from (a).
Let $\gamma_i = \max\{2\a_i + 1, 0\}$, $i=1,\ldots,d$. Then, obviously 
\[
z^{2\a + \1} \le z^\gamma, \qquad z \in \Rr,
\]
and an application of Lemma \ref{lem:L7L9L10} (c) with $b = \infty$ produces
\begin{align*}
\nu_\a \big\{z \in \Rr : |z| \ge a \big\}
& \le 
\int_{|z| \ge a} z^{\gamma} e^{-|z|^2} \, dz
\simeq
a^{|\gamma|+ d-2} e^{-a^2}
\lesssim
e^{-3a^2/4},
\end{align*}
as desired.
\end{proof}

\begin{proof}[{Proof of Lemma \ref{lem:p_gg1}}]
We first show that there exists $c>0$ such that
\begin{align}\label{est12}
\sup_{1/4<s<1} \mathbb{K}_s^{\a}(x,y)
\lesssim
x^{-2\a - \1} \exp\big(|x|^2 - c |x|^2 \t^2\big), 
\qquad x \in \Rr, \quad y \in \RR,
\end{align}
where $\t = \t(x,y)$.

Taking into account the fact that $1/4<s<1$ and then using Lemma~\ref{lem:L3_2}
(with $\gamma = \a_i$, $a = (1+s)x_i \simeq x_i \ge 1$,  $b = (1-s)y_i$, $C = 1/16$, 
$i = 1, \ldots, \e$), we infer that
\begin{align*}
\mathbb{K}_s^{\a}(x,y)
& \lesssim
(1-s)^{-2 \langle \a' \rangle}
(x' y')^{-2\a'}
(x')^{2\a' - \1'} (x'')^{-2\a'' - \1''}
\exp\bigg( |x|^2 - \frac{|(1+s)x-(1-s)y|^2}{8}\bigg) \\
& \lesssim
x^{-2\a - \1} \exp\bigg( |x|^2 - \frac{|(1+s)x-(1-s)y|^2}{16}\bigg),
\end{align*}
uniformly in $1/4<s<1$, $x \in \Rr$ and $y \in \RR$.
Finally, applying \eqref{est11} we obtain \eqref{est12}. 

Observe that in case $d=1$ the conclusion easily follows because from \eqref{est12}
we see that the relevant kernel is controlled by 
$\chi_{ \{ x \ge 1 \} } x^{-2\a - 1}e^{x^2}$, and this function of $x$ is in weak $L^1(d\nu_{\a})$,
even with some margin. Thus from now on we assume that $d \ge 2$. 

In order to finish the proof of Lemma~\ref{lem:p_gg1} it suffices to show that for $\a \in \R^d$
and any fixed $c > 0$ the operator
\[
V^{\a} f (x) 
=
\chi_{\Rr}(x)\, x^{-2\a - \1} e^{|x|^2} 
\int_{\RR} \exp( - c |x|^2 \t^2) f(y) \, d\nu_{\a}(y),
\qquad x \in \RR, \quad 0 \le f \in L^1(d\nu_{\a}),
\]
is of weak type $(1,1)$ with respect to $\nu_{\a}$; here $\t = \t(x,y) \in [0,\pi/2]$.
We will proceed in a similar way as in the proof of the weak type $(1,1)$ estimate for $U^{\a}$ 
in Section \ref{sec:Lc_glob}.

Let $\lam > 0$ and $0 \le f \in L^1(d\nu_\a)$. Let $z_0 \in \Rr$ be such that
\[
|z_0| = \min \big\{|z|: z \in \Rr, \,\, 
V^{\a} f(z) \ge   \lambda  \big\}, 
\qquad r_0 := |z_0| \ge 1.
\]
The definition of $z_0$ is correct because the level set above is closed in $\Rr$
(the function $x \mapsto V^{\a} f(x)$ is continuous on $\Rr$) and we may assume that it is nonempty. 
This forces $V^{\a} f(z_0) \ge \lambda$, which in particular means that
\[ 
z_0^{-2\a - \1} e^{r_0^2} 
\|f\|_{L^1(d\nu_\a)} \ge \lam.
\]
Combining this with Lemma~\ref{lem:F8&L7L9L10_2} (b) (taken with $a = 2r_0 \ge 2$) we get
\begin{align*}
\nu_\a  \big\{ z \in \Rr : |z| \ge 2r_0 \big\}
& \lesssim 
e^{-3r_0^2} 
\le 
\frac{\|f\|_{L^1(d\nu_\a)} }{ \lam }
z_0^{-2\a - \1} e^{-2r_0^2} 
\lesssim
\frac{\|f\|_{L^1(d\nu_\a)} }{ \lam },
\end{align*}
the last estimate above since $1 \le (z_0)_i \le |z_0| = r_0$, $i = 1, \ldots, d$.
Thus we reduced our considerations to the region 
$r_0 \le |z| \le 2 r_0$.

Define 
\[
H = \big\{ w \in S_+^{d-1} : \textrm{there exists $\rho \in [r_0, 2r_0]$
	such that $V^{\a} f(\rho w) \ge \lam$} \big\},
\]
and for every $w \in H$ let
\[
r(w) = \min \left\{ \rho \in [r_0, 2r_0] : V^{\a} f(\rho w) \ge \lam \right\}.
\]
By continuity, 
the minimum exists and
$V^{\a} f(r(w) w) \ge \lam$ for $w \in H$. Further, since $r(w) \simeq r_0$, we see that 
\begin{align*}
r_0^{-2 \langle \a \rangle - d} w^{-2\a - \1} e^{r(w)^2} 
\int_{\RR} 
\exp(- {c} r_0^2 \t^2) f(y) \, d\nu_\a (y) \gtrsim \lam,
\end{align*}
where $\t = \t(w,y)$.

Using Lemma~\ref{lem:F8&L7L9L10_2} (a) (specified to
$\gamma = 2 \langle \a \rangle + 2 d - 1$, $a = r(w) \simeq r_0 \ge 1$)
and then the above estimate, we arrive at
\begin{align*}
& \nu_\a \big\{ z \in \Rr : r_0 \le |z| \le 2r_0, \, \, V^{\a} f (z) \ge \lam \big\} \\
& \le 
\nu_\a \big\{ rw : w \in H, \, \, r \ge r(w) \big\} 
=
\int_H \int_{r(w)}^\infty r^{2 \langle \a \rangle + 2d - 1} e^{-r^2} 
\, dr \, w^{2\a + \1} \, d\sigma(w) \\
& \simeq 
\int_H r_0^{2 \langle \a \rangle + 2d - 2} e^{-r(w)^2} w^{2\a + \1} \, d\sigma(w) \\
& \lesssim
\lam^{-1} \int_H r_0^{d-2} \int_{\RR}
\exp(- {c} r_0^2 \t^2)
f(y) \, d\nu_\a (y) \, d\sigma(w).
\end{align*}
To finish the proof of the weak type $(1,1)$ for $V^{\a}$, it is enough to check that
\[
\int_{S^{d-1}} \exp(- {c} r_0^2 d(w,\tilde{y})^2) \, d\sigma(w)
\lesssim r_0^{2 - d} ,
\qquad  y \in \RR, \quad r_0 \ge 1.
\]
This integral is comparable to $\int \exp(- {c} r_0^2 |\xi|^2) \, d\xi$ taken over the unit ball
in $\mathbb{R}^{d-1}$. It is thus bounded by a constant times $r_0^{1 - d}$.

The proof of Lemma \ref{lem:p_gg1} is complete.
\end{proof}

Now Theorem \ref{thm:weakr} and thus also Theorem \ref{thm:weak} are proved.

\section{Exotic Bessel semigroup maximal operator} \label{sec:Bes}

To start with, we focus on the one-dimensional situation.

\subsection{Description of the classical and exotic Bessel contexts in dimension one}
Recall from Section \ref{sec:intro} that the Bessel differential operator is 
\[
B_{\nu} = -\frac{d^2}{dx^2} - \frac{2\nu+1}{x} \frac{d}{dx}.
\]
For a given $\nu \in \R$ we consider $B_{\nu}$ acting on functions on $\R_+$.
This operator is formally symmetric in $L^2(d\eta_{\nu})$, where, according to the
notation of Section \ref{sec:Lagcls}, 
\[
d\eta_{\nu} (x) = x^{2\nu + 1} \, dx, \qquad x > 0.
\]

When $\nu > -1$, there exists a classical self-adjoint extension of $B_{\nu}$ (acting initially on
$C_c^{2}(\R_+)$), from now on denoted by $B_{\nu}^{\textrm{cls}}$, whose spectral decomposition
is given via the (modified) Hankel transform. To make this more precise,
consider for each $z > 0$ and $\nu > -1$ the function 
\[
\vp_z^{\nu} (x) = (xz)^{-\nu} J_{\nu}(xz), \qquad x>0, 
\]
where $J_{\nu}$ denotes the Bessel function of the first kind and order $\nu> -1$.
It is well known that $\vp_z^{\nu}$ is an eigenfunction of $B_{\nu}$ with the corresponding eigenvalue $z^2$,
that is to say $B_{\nu} \vp_z^{\nu} = z^2 \vp_z^{\nu}$.
The (modified) Hankel transform is defined initially for, say, $f \in C_c (\R_+)$ as
\[
h_\nu f(z) = \int_0^\infty f(x) \varphi_z^{\nu} (x) \, d\eta_\nu (x), 
\qquad z> 0, \quad \nu > -1,
\]
and plays in the Bessel context a role similar to that of the Fourier transform in the Euclidean setup.
It is well known that $h_\nu$ extends to an isometry on $L^2(d\eta_{\nu})$ and $h_\nu^{-1} = h_\nu$.
Further, for $f \in C_c^{2}(\R_+)$ we have
\[
h_\nu \big( B_{\nu} f \big) (z) =  z^2  h_\nu f (z), \qquad z> 0.
\]
Therefore the classical self-adjoint extension of $B_{\nu}$ is defined by
\[
B_{\nu}^{\textrm{cls}} f = h_\nu \big( z^2  h_\nu f (z) \big),
\]
on the domain $\dom B_{\nu}^{\textrm{cls}}$ consisting of all $f \in L^2(d\eta_\nu)$ such that
$z^2  h_\nu f (z) \in L^2(d\eta_\nu)$.

The classical Bessel semigroup $W_t^\nu = \exp(-t B_{\nu}^{\textrm{cls}})$
is given in $L^2(d\eta_\nu)$ by the spectral formula
\[
W_t^\nu f = h_\nu \big(e^{- t z^2}  h_\nu f (z) \big),
\qquad t \ge 0.
\]
This semigroup has in $L^2(d\eta_\nu)$ the integral representation
$$
W_t^{\nu} f (x) = 
\int_0^\infty W_t^{\nu}(x,y) f(y) \, d\eta_\nu (y), 
\qquad x,t>0,
$$
with the integral kernel 
\begin{align*}
W_t^{\nu}(x,y) & = 
\int_0^\infty e^{- t z^2} \vp_z^{\nu} (x) \vp_z^{\nu} (y) \, d\eta_\nu (z) \\
& =
\frac{1}{2t} \exp\Big( -\frac{1}{4t}(x^2 + y^2) \Big)
(xy)^{-\nu} I_{\nu}\Big( \frac{xy}{2t}\Big), \qquad x,y,t > 0.
\end{align*}
Using standard properties of the modified Bessel function, we see that $W_t^{\nu}(x,y)$
is strictly positive and smooth in $(x,y,t) \in \R^3_+$.
Further, an application of \eqref{Basy*} shows that
the integral defining $W_t^{\nu} f$ converges absolutely for any 
$f \in L^p(d\eta_\nu)$, $1 \le p \le \infty$, and thus provides 
a pointwise definition of $W_t^{\nu} f$, $t>0$, for all 
$f \in L^p(d\eta_\nu)$, $1 \le p \le \infty$.
From \cite[Proposition 6.2]{NS} we know that $\{W_t^{\nu}\}$ is a Markovian symmetric diffusion semigroup.
In particular, $W_t^{\nu} \mathds{1} = \mathds{1}$ and $\{ W_t^{\nu} \}$ is a positive and symmetric
semigroup of contractions on each $L^p(d\eta_\nu)$, $1 \le p \le \infty$.

We now pass to the exotic situation, which occurs when $0 \ne \nu < 1$.
It turns out that for these $\nu$ there exists a self-adjoint extension of $B_{\nu}$ (considered
initially on $C_c^{2}(\R_+)$) expressible in terms of the (modified) Hankel transform, but
in a different way than $B_{\nu}^{\textrm{cls}}$. In order to describe the details, observe first
that for each $z > 0$ the function $x \mapsto (x z)^{-2\nu} \vp_z^{-\nu}(x)$ is an eigenfunction of
$B_{\nu}$ with the corresponding eigenvalue $z^2$. Next, we introduce the exotic (modified) Hankel
transform defined initially for, say, $f \in C_c (\R_+)$ as
\begin{align*}
\widetilde{h}_\nu f(z) = \int_0^\infty f(x) (x z)^{-2\nu} \vp_z^{-\nu} (x) \, d\eta_\nu (x) 
= z^{-2\nu} h_{- \nu} \big(x^{2\nu} f(x) \big)(z), 
\qquad z> 0. 
\end{align*}
Using the above connection with the classical Hankel transform, we see that $\widetilde{h}_\nu$ inherits
some properties of $h_{\nu}$. In particular, it extends to an isometry on $L^2(d\eta_\nu)$ and
$\widetilde{h}_\nu^{-1} = \widetilde{h}_\nu$.
Further, a computation shows that $x^{2\nu} B_{\nu} f = B_{- \nu} \big(x^{2\nu} f \big)$, which implies
$\widetilde{h}_\nu \big( B_{\nu} f \big) (z) =  z^2  \widetilde{h}_\nu f (z)$, for $z>0$ and
$f \in C_c^{2}(\R_+)$. This leads us to define the exotic self-adjoint extension of $B_{\nu}$ as
\[
B_{\nu}^{\textrm{exo}} f 
= \widetilde{h}_\nu \big( z^2 \widetilde{h}_\nu f (z) \big),
\]
on the domain $\dom B_{\nu}^{\textrm{exo}}$ consisting of all $f \in L^2(d\eta_\nu)$ such that 
$z^2  \widetilde{h}_\nu f (z) \in L^2(d\eta_\nu)$.

The exotic Bessel semigroup 
$\widetilde{W}_t^\nu = \exp(-t B_{\nu}^{\textrm{exo}})$ 
generated by $- B_{\nu}^{\textrm{exo}}$ is given in $L^2(d\eta_\nu)$ by the spectral formula
\[
\widetilde{W}_t^\nu f = 
\widetilde{h}_\nu \big(e^{- t z^2}  \widetilde{h}_\nu f (z) \big),
\qquad t \ge 0.
\]
The related integral representation is 
$$
\widetilde{W}_t^{\nu} f (x) = 
\int_0^\infty \widetilde{W}_t^{\nu}(x,y) f(y) \, d\eta_\nu (y), 
\qquad x,t > 0,
$$
where the integral kernel is expressed as
\begin{align} \nonumber
\widetilde{W}_t^{\nu}(x,y) & = 
\int_0^\infty e^{- t z^2} (x z)^{-2\nu} \vp_z^{-\nu} (x) \, (y z)^{-2\nu} \vp_z^{-\nu} (y) 
\, d\eta_\nu (z) \\ \label{relBker}
& =
(x y)^{-2\nu} W_t^{- \nu}(x,y), \qquad x,y,t > 0.
\end{align}
Since the exotic kernel can be expressed in a simple way in terms of the classical one, it inherits some
properties of the latter kernel. In particular, $\widetilde{W}_t^{\nu}(x,y)$ is strictly positive and smooth
in $(x,y,t) \in \R_+^3$. Further, using \eqref{Basy*} it is straightforward to check that for $\nu < 0$
the integral defining $\widetilde{W}_t^{\nu} f$ converges absolutely for every $f \in L^p(d\eta_{\nu})$,
$1 \le p \le \infty$. Furthermore, for such $\nu$ the operators $\{ \widetilde{W}_t^{\nu} \}$ satisfy the
semigroup property on each $L^p(d\eta_{\nu})$, $1 \le p \le \infty$, which is a direct consequence of
\eqref{relBker}. 
The case $0 < \nu < 1$ is more subtle, as in the Laguerre context.
Indeed, in this range of $\nu$ we have a pencil type phenomenon. More precisely, for each $t > 0$ fixed, 
$\widetilde{W}_{t}^{\nu}$ is well defined on $L^p(d\eta_{\nu})$
and maps this space into itself if and only if $\nu+1 < p < (\nu+1)/\nu$. 

Observe that in view of \eqref{iI} we have the bound
\begin{equation} \label{denB}
	\widetilde{W}_t^{\nu}(x,y) < W_t^{\nu}(x,y), \qquad x,y,t>0, \quad -1 < \nu < 0.
\end{equation}
Moreover, note that the classical and exotic Bessel settings have probabilistic interpretations analogous
to those in the Laguerre case, see Section \ref{sec:intro} and also \cite[Appendix 1]{BS}.

Another interesting fact is that for $\nu < 0$ the operators 
$\widetilde{W}_t^{\nu}$, $t>0$, are contractive on $L^\infty$. We even have the strict estimate
\begin{equation} \label{bescontr}
\widetilde{W}_t^{\nu} \mathds{1}(x) < 1, \qquad x,t>0, \quad \nu < 0.
\end{equation}
Indeed, proceeding as in the Laguerre context (see Section \ref{sec:Lagexo})
and using the explicit form of $\widetilde{W}_t^{\nu}(x,y)$ we get
$$
\widetilde{W}_t^{\nu} \mathds{1} (x) = 
	\frac{1}{2t} x^{-\nu} \exp\Big( -\frac{x^2}{4t}\Big) \, \mathcal{J}(x),
$$
where
$$
\mathcal{J}(x) 
	= \int_0^{\infty} y^{\nu+1} \exp\Big( -\frac{y^2}{4t}\Big) \, I_{-\nu}\Big(\frac{xy}{2t}\Big) \, dy 
  = \frac{2^{2\nu+1}}{\Gamma(1-\nu)} t^{\nu+1}x^{-\nu} {_1F_1}\Big( 1; 1-\nu; \frac{x^2}{4t}\Big);
$$
here the last identity is obtained by means of \cite[Lemma 2.2]{NS}. Consequently, with the notation
of \cite[Section 2]{NS},
$$
\widetilde{W}_t^{\nu} \mathds{1} (x) = H_{1,1-\nu}\Big( \frac{x^2}{4t} \Big), \qquad x,t > 0.
$$
By the proof of \cite[Lemma 2.3]{NS}, $\widetilde{W}_t^{\nu} \mathds{1} (x)< 1$ for $x,t > 0$ and
$\|\widetilde{W}_t^{\nu} \mathds{1} \|_{\infty} = 1$ for all $t > 0$.

\subsection{Multi-dimensional exotic Bessel context and the maximal theorem}
Now we are ready to introduce the multi-dimensional framework, which arises by `tensorizing'
the one-dimensional classical and exotic Bessel settings. 
Let $d \ge 1$ and $\nu \in \R^d$ be a multi-parameter, and recall that
\[
d\eta_{\nu} (x) = x^{2\nu + \1} \, dx, \qquad x \in \RR.
\]
In what follows we assume that $\nu \in A(\E)$ for some fixed $\E \subset \{1,\ldots,d\}$,
see \eqref{def:ABC}.

For each $z \in \RR$ define
\[
\Phi_z^{\nu,\E}  = 
\bigotimes_{i=1}^d 
\begin{cases}
(x_i z_i )^{-2 \nu_i} \vp_{z_i}^{- \nu_i} , & \quad i \in \E,\\
\vp_{z_i}^{\nu_i} , & \quad i \notin \E.
\end{cases}
\]
These are eigenfunctions of the multi-dimensional Bessel differential operator
$\mathbb{B}_{\nu} = \sum_{i=1}^d B_{\nu_i}$ (here $B_{\nu_i}$ acts on the $i$th coordinate variable)
with the corresponding eigenvalues $|z|^2$.
For $f \in C_c(\RR)$ the generalized (modified) Hankel transform is given by
\[
\mathfrak{h}_{\nu,\E} f(z) = \int_{\RR} f(x) \Phi_z^{\nu,\E} (x) \, d\eta_\nu (x), 
\qquad z \in \RR. 
\]
Using properties of the one-dimensional
Hankel and exotic Hankel transforms, it can easily be justified that $\mathfrak{h}_{\nu,\E}$
extends to an isometry on $L^2(d\eta_{\nu})$ and coincides with its inverse,
$\mathfrak{h}_{\nu,\E}^{-1} = \mathfrak{h}_{\nu,\E}$.

We consider the self-adjoint extension of $\mathbb{B}_{\nu}$ acting initially on $C_c^{2}(\RR)$ given by
\[
\mathbb{B}_{\nu,\E} f =  \mathfrak{h}_{\nu,\E} \big( |z|^2  \mathfrak{h}_{\nu,\E} f (z) \big),
\]
on the domain $\dom \mathbb{B}_{\nu,\E}$ consisting of all $f \in L^2(d\eta_\nu)$ such that
$|z|^2  \mathfrak{h}_{\nu,\E} f (z) \in L^2(d\eta_\nu)$. Observe that for
$\E = \emptyset$ we recover the classical multi-dimensional Bessel context considered in
\cite{BCC0,BCC,BCDFR,BCN,CaSz}, among many other papers. Otherwise, that is when $\E \neq \emptyset$,
the exotic situation occurs.

The semigroup $\mathbb{W}_t^{\nu,\E} = \exp(-t \mathbb{B}_{\nu,\E} )$ is given in 
$L^2(d\eta_{\nu})$ via the $\mathfrak{h}_{\nu,\E}$,
\[
\mathbb{W}_t^{\nu,\E} f = 
\mathfrak{h}_{\nu,\E} \big( e^{-t |z|^2} \mathfrak{h}_{\nu,\E} f (z) \big),
\qquad t \ge 0.
\]
Further, it has in $L^2(d\eta_{\nu})$ the integral representation
\begin{align}\label{eBir}
\mathbb{W}_t^{\nu,\E} f (x) = 
\int_{\RR} \mathbb{W}_t^{\nu,\E} (x,y) f(y) \, d\eta_\nu (y), 
\qquad x \in \RR, \quad t>0,
\end{align}
where the kernel is a product of the one-dimensional kernels,
$$
\mathbb{W}_t^{\nu,\E}(x,y) = 
	\prod_{i \in \E} \widetilde{W}_t^{\nu_i}(x_i,y_i)
	\prod_{i \in \E^c} W_t^{\nu_i}(x_i,y_i), \qquad  x,y \in \R^d_+, \quad t>0.
$$
Obviously, $\mathbb{W}_t^{\nu,\E}(x,y)$ is strictly positive and smooth
in $(x,y,t) \in \R^{2d+1}_+$, which
follows from the analogous properties of $\widetilde{W}_t^{\nu_i}(x_i,y_i)$ and $W_t^{\nu_i}(x_i,y_i)$.

As in the Laguerre context, when $\mE(\nu) < 0$ the integral formula \eqref{eBir} makes sense
and provides a pointwise definition of $\mathbb{W}_t^{\nu,\E} f$ on all $L^p(d\eta_{\nu})$ spaces,
$1\le p \le \infty$. Further, by \eqref{bescontr}
we have the inequality $\mathbb{W}_t^{\nu,\E} \mathds{1}(x) \le 1$,
$x \in \R^d_{+}$, $t>0$, which is strict if $\E \neq \emptyset$ (for $\E = \emptyset$ one
has $\mathbb{W}_t^{\nu,\emptyset} \mathds{1} = \mathds{1}$).
This shows that $\{\mathbb{W}_t^{\nu,\E}\}$ is a semigroup of contractions on each
$L^p(d\eta_{\nu})$, $1\le p \le \infty$, and consequently 
it is a submarkovian symmetric diffusion semigroup, which is Markovian if and only if $\E = \emptyset$.

In case $\mE(\nu) > 0$ a pencil type phenomenon occurs. More precisely, for each $t > 0$
fixed, $\mathbb{W}_{t}^{\nu,\E} f$ is defined on $L^p(d\eta_{\nu})$
and maps this space into itself if and only if 
$1 + \mE(\nu) < p < 1 + 1/\mE(\nu)$.

The principal object of our study in this section is the maximal operator
$$
\mathbb{W}_{*}^{\nu,\E}f = \sup_{t>0} \big| \mathbb{W}_{t}^{\nu,\E}f\big|.
$$
We aim at showing the weak type $(1,1)$ estimate for $\mathbb{W}_{*}^{\nu,\E}$.
This question makes sense only for $\nu$ satisfying 
$\mE(\nu) < 0$, in view of the above discussion concerning the pencil type phenomenon. 
It is worth pointing out that for such $\nu$ and $p>1$ the operator $\mathbb{W}_{*}^{\nu,\E}$ is
$L^p(d\eta_{\nu})$-bounded, by Stein's maximal theorem \cite[p.\,73]{St}.
On the other hand, when $\mE(\nu) > 0$ the maximal operator $\mathbb{W}_{*}^{\nu,\E}$
is not even defined on $L^1(d\eta_{\nu})$.

The following theorem is our main result in the Bessel setting.
\begin{thm} \label{thm:weakB}
Let $d \ge 1$ and $\nu \in A(\E)$ for some $\E \subset \{1,\ldots,d\}$. Assume that $\mE(\nu) < 0$. 
Then $\mathbb{W}_{*}^{\nu,\E}$ is bounded from $L^1(d\eta_{\nu})$ to weak $L^{1}(d\eta_{\nu})$.
\end{thm}

In the special case $\E = \emptyset$, Theorem \ref{thm:weakB} says that the classical
multi-dimensional Bessel semigroup maximal operator is of weak type $(1,1)$. This result is already
well known, see for instance \cite[Theorem 1.1]{BCC}, \cite[Theorem 1.1]{BCDFR}, \cite[Theorem 2.1]{BCN}
or \cite[Theorem 2.1]{CaSz}. Nevertheless, we take this opportunity to present a short and direct
argument: since the kernel $\mathbb{W}_t^{\nu,\emptyset}(x,y)$ possesses the so-called
Gaussian bound related to the space of homogeneous type $(\R_+^d,\eta_{\nu},|\cdot|)$,
see \cite[Lemma 4.2]{DPW}, the weak type $(1,1)$ estimate for $\mathbb{W}_*^{\nu,\emptyset}$
follows from the general theory.
Actually, to verify the Gaussian bound it is enough to use \eqref{Basy*} and 
Lemma \ref{lem:L3} (with $\gamma = - \nu_i -1/2$, $C=1/8$, $a = \frac{x_i}{\sqrt{t}}$ and
$b = \frac{ y_i}{\sqrt{t}}$, $i=1, \ldots, d$) to get
\begin{align}\label{Bclcomp}
\mathbb{W}_{t}^{\nu,\emptyset} (x,y)
& \simeq 
t^{-d/2} \prod_{i=1}^d \big(x_i y_i + t \big)^{-\nu_i-1/2}
	\exp\bigg( -\frac{|x-y|^2}{4t}\bigg) \\ \nonumber
& \lesssim
t^{-d/2} \prod_{i=1}^d \big(x_i +\sqrt{t}\,\big)^{-2\nu_i-1}
	\exp\bigg( -\frac{|x-y|^2}{8t}\bigg), 
\qquad x,y \in \RR, \quad t > 0,
\end{align}
and then combine the last estimate with \eqref{Bball}.

When proving Theorem \ref{thm:weakB}, we can make analogous reductions to those described in items (R1)--(R5)
following the statement of Theorem \ref{thm:weak}. Indeed, instead of \eqref{den} and Theorem
\ref{thm:max_Lcl} we use \eqref{denB} and the weak type $(1,1)$ of $\mathbb{W}_*^{\nu,\emptyset}$,
respectively. Therefore we may assume that 
$\nu = (\nu',\nu'') \in (-\infty,-1]^{\e} \times (-1,\infty)^{d-\e}$
for some $1 \le \e \le d$, and that the kernel is given by
$$
\mathbb{W}_t^{\nu}(x,y) 
= (x'y')^{-2\nu'} \mathbb{W}_t^{-\nu',\emptyset}(x',y') \mathbb{W}_t^{\nu'',\emptyset}(x'',y''),
\qquad x,y \in \RR, \quad t > 0,
$$
where, as before, for $z \in \R^d$ we denote $z'=(z_1,\ldots,z_{\e})\in \R^{\e}$ and 
$z''=(z_{\e+1},\ldots,z_d)\in \R^{d-\e}$.
Now, having in mind the product structure of $\mathbb{W}_t^{\nu}(x,y)$, we see that the proof of
Theorem \ref{thm:weakB} boils down to showing the following two lemmas.

\begin{lem} \label{lem:p_locB}
For any $d \ge 1$, $1 \le \e \le d$ and each $\nu \in (-\infty,-1]^{\e} \times (-1,\infty)^{d-\e}$
the maximal operator
$$
\mathbb{W}^{\nu}_{*,1}f(x) = \sup_{t > 0} \int_{\R^d_+} 
	\chi_{\{x_i/2 \le y_i \le 2x_i\; \textrm{for}\; i=1,\ldots,\e\}} \mathbb{W}^{\nu}_t(x,y)
	 f(y)\, d\eta_{\nu}(y), \qquad f \ge 0,
$$
is of weak type $(1,1)$ with respect to $\eta_{\nu}$.
\end{lem}

\begin{proof}
The proof is just a repetition of the arguments used in the proof of Lemma \ref{lem:p_loc},
where instead of Theorem \ref{thm:max_Lclr} one should use the weak type $(1,1)$ of
the classical Bessel semigroup maximal operator.
\end{proof}

\begin{lem} \label{lem:p_glob_0B}
For each $\beta \le -1$ the one-dimensional operator
$$
N^{\beta}f(x) = \int_{0}^{\infty} \chi_{\{y < x/2 \;\textrm{or}\; y>2x\}} 
	\bigg[(xy)^{-2\beta} \sup_{t > 0} 
	 W_t^{-\beta}(x,y)\bigg] f(y)\, d\eta_{\b}(y), \qquad f \ge 0,
$$
is of strong type $(1,1)$ with respect to $\eta_{\beta}$.
\end{lem}

\begin{proof}
Using \eqref{Basy*} we see that the kernel of $N^{\beta}$ is comparable with
\[
N (x,y) = \chi_{\{y < x/2 \;\textrm{or}\; y>2x\}}
(xy)^{-2\b} \sup_{t > 0} \,(xy + t)^{\b - 1/2} \frac{1}{\sqrt{t}} 
\exp\left( -\frac{|x-y|^2}{4t} \right).
\]
Further, the constraint $\{y < x/2$ or $y>2x\}$ implies $|y-x| > (x \vee y)/2$. Then, 
using Lemma~\ref{lem:Gauss2} (with $\kappa = -1/2$, $\gamma = \b - 1/2 \le 0$, $c=1/16$,
$A = xy$ and $z = x \vee y$) it is easy to check that 
\begin{align*}
N (x,y)
& \lesssim
(xy)^{-2\b} \sup_{t > 0}\, (xy + t)^{\b - 1/2} \frac{1}{\sqrt{t}} 
\exp\left( -\frac{(x \vee y)^2}{16t} \right) \\
& \simeq 
(xy)^{-2\b} (x \vee y)^{2\b-2}
= (x \wedge y)^{-2\b} (x \vee y)^{-2}.
\end{align*}
To conclude it suffices to verify that
\[
\int_0^{\infty} (x \wedge y)^{-2\b} (x \vee y)^{-2} \, d\eta_{\b}(x)
\lesssim
1, \qquad y > 0,
\]
which is trivial.
\end{proof}
Now Theorem \ref{thm:weakB} is proved.

\section{Exotic Jacobi semigroup maximal operator} \label{sec:Jac}

For the sake of clarity we first describe the one-dimensional situation. 
\subsection{Description of the classical and exotic Jacobi contexts in dimension one}
Recall from Section \ref{sec:intro} that the Jacobi differential operator is given by
$$
J_{\ab} = -(1-x^2) \frac{d^2}{dx^2} - \big[\beta-\alpha - (\alpha+\beta+2)x\big]\frac{d}{dx}.
$$
Here $\ab \in \R$ are the type parameters and we consider $J_{\ab}$ acting on functions on the
interval $(-1,1)$. A natural measure $\rho_{\ab}$ in $(-1,1)$ associated with $J_{\ab}$ has the form
$$
d\rho_{\ab}(x) = (1-x)^{\a}(1+x)^{\b}\, dx.
$$
From the factorization
$$
J_{\ab}f(x) = - \big[(1-x)^{\a}(1+x)^{\b}\big]^{-1} \frac{d}{dx} \Big( (1-x)^{\a+1}(1+x)^{\b+1}
	\frac{d}{dx}f(x)\Big)
$$
it is readily seen that $J_{\ab}$ is formally symmetric in $L^2(d\rho_{\ab})$.

When $\a > -1$ and $\b > -1$, there exists a classical self-adjoint extension of $J_{\ab}$
(acting initially on $C_c^{2}(-1,1)$), from now on denoted by $J_{\ab}^{\textrm{cls,cls}}$, whose
spectral decomposition is given by the Jacobi polynomials $P_n^{\ab}$, $n=0,1,2,\ldots$.
The latter form an orthogonal basis in $L^2(d\rho_{\ab})$ and one has
$J_{\ab}P_n^{\ab} = n(n+\a+\b+1)P_n^{\ab}$. Thus
$$
J_{\ab}^{\textrm{cls,cls}}f = \sum_{n=0}^{\infty} n(n+\a+\b+1) 
	\big\langle f, \breve{P}_n^{\ab}\big\rangle_{d\rho_{\ab}} \, \breve{P}_n^{\ab},
$$
on the domain $\dom J_{\ab}^{\textrm{cls,cls}}$ consisting of all $f\in L^2(d\rho_{\ab})$ for which
this series converges in $L^2(d\rho_{\ab})$;
here $\breve{P}_n^{\ab} = P_n^{\ab}/\|P_n^{\ab}\|_{L^2(d\rho_{\ab})}$ are the normalized Jacobi polynomials.

The classical Jacobi semigroup $T_t^{\ab} = \exp(-tJ_{\ab}^{\textrm{cls,cls}})$, $t \ge 0$,
has in $L^2(d\rho_{\ab})$ the integral representation
\begin{equation} \label{int_Jac}
T_t^{\ab}f(x) = \int_{-1}^1 G_t^{\ab}(x,y)f(y)\, d\rho_{\ab}(y), \qquad x \in (-1,1), \quad t>0,
\end{equation}
where
$$
G_t^{\ab}(x,y) = \sum_{n=0}^{\infty} e^{-tn(n+\a+\b+1)} \breve{P}_{n}^{\ab}(x) \breve{P}_n^{\ab}(y),
	\qquad x,y \in (-1,1), \quad t>0.
$$
Although no explicit formula is known for $G_t^{\ab}(x,y)$, the following qualitatively sharp estimates of
this kernel were established recently; see \cite[Theorem A]{NoSj} and \cite[Theorem 7.2]{CKP}.
\begin{thm}[{\cite{CKP, NoSj}}] \label{jkerest}
Let $\ab > -1$. There exist constants $c_1,c_2>0$ such that
\begin{align*}
& \big[t \vee \theta \varphi\big]^{-\a-1/2}
	\big[ t \vee (\pi-\theta)(\pi-\varphi)]^{-\b-1/2}
	\frac{1}{\sqrt{t}} \exp\bigg( - c_1\frac{(\theta-\varphi)^2}{t}\bigg)\\
& \lesssim
G_{t}^{\ab}(\cos\theta,\cos\varphi) \lesssim \big[t \vee \theta \varphi\big]^{-\a-1/2}
	\big[ t \vee (\pi-\theta)(\pi-\varphi)]^{-\b-1/2}
	\frac{1}{\sqrt{t}} \exp\bigg( - c_2\frac{(\theta-\varphi)^2}{t}\bigg),
\end{align*}
uniformly in $\theta,\varphi \in (0,\pi)$ and $0 < t \le 1$. Furthermore,
$$
G_t^{\ab}(x,y) \simeq 1, \qquad x,y \in (-1,1), \quad t \ge 1.
$$
\end{thm}
From Theorem \ref{jkerest} we see that $G_t^{\ab}(x,y)$ is strictly positive, while smoothness in
$(x,y,t) \in (-1,1)^2\times \R_+$ can be deduced from the series representation
(see e.g.\ \cite[Section 2]{NoSj}). Moreover, \eqref{int_Jac} provides a pointwise definition of
$T_t^{\ab}f$ (the defining integral converges absolutely) for $f \in L^1(d\rho_{\ab})$ and thus for
any $f \in L^p(d\rho_{\ab}) \subset L^1(d\rho_{\ab})$, $1 \le p \le \infty$.
Since $P_0^{\ab}$ is constant, we have $T_t^{\ab}\mathds{1} = \mathds{1}$, and consequently
$\{T_t^{\ab}\}$ is a Markovian symmetric diffusion semigroup. In particular, $T_t^{\ab}$, $t>0$,
are contractions on each $L^p(d\rho_{\ab})$, $1 \le p \le \infty$.

Inspired by Hajmirzaahmad \cite{H1}, we now pass to exotic situations. Roughly speaking, they occur in three 
essentially different cases depending on whether only
one or both type parameters are exotic. More precisely, we distinguish the following exotic situations:
\begin{itemize}
\item
$0 \neq \a < 1$ and $\b > -1$ and only $\a$ is exotic,
\item
$\a > -1$ and $ 0 \neq \b < 1$ and only $\b$ is exotic,
\item
$0 \neq \a < 1$ and $0 \neq \b < 1$ and both $\a$ and $\b$ are exotic.
\end{itemize}
For the sake of clarity, we look at each case separately.

Assume that only $\a$ is exotic and so let $0 \neq \a < 1$ and $\b > -1$. Consider the system
$\{(1-x)^{-\a}P_n^{-\a,\b} : n \ge 0\}$. This is an orthogonal basis in $L^2(d\rho_{\ab})$ and we have
$$
J_{\ab} \big( (1-x)^{-\a} P_n^{-\a,\b} \big) = \big[ n(n-\a+\b+1) - \a (\b+1)\big] (1-x)^{-\a}P_n^{-\a,\b},
	\qquad n \ge 0.
$$
The last two facts follow from the analogous ones for the system of Jacobi polynomials in the non-exotic
case. Observe that $(1-x)^{-\a}P_n^{-\a,\b}/\|(1-x)^{-\a}P_n^{-\a,\b}\|_{L^2(d\rho_{\ab})}
= (1-x)^{-\a} \breve{P}_n^{-\a,\b}$. We define the emerging self-adjoint extension of $J_{\ab}$ by
$$
J_{\ab}^{\textrm{exo,cls}}f = \sum_{n=0}^{\infty} \big[ n(n-\a+\b+1) - \a (\b+1)\big]
	\big\langle f, (1-x)^{-\a}\breve{P}_n^{-\a,\b} \big\rangle_{d\rho_{\ab}} \, (1-x)^{-\a}\breve{P}_n^{-\a,\b}
$$
on the domain $\dom J_{\ab}^{\textrm{exo,cls}}$ consisting of all $f \in L^2(d\rho_{\ab})$ for which the
above series converges in $L^2(d\rho_{\ab})$. It is easy to check that $J_{\ab}^{\textrm{exo,cls}}$
is non-negative in the spectral sense if and only if $\a < 0$.

The exotic Jacobi semigroup
${\overset{\sim,\cdot}{T}}{}_{\! t}^{\a,\b} = \exp(-t J_{\ab}^{\textrm{exo,cls}})$
is given in $L^2(d\rho_{\ab})$ by
$$
{\overset{\sim,\cdot}{T}}{}_{\! t}^{\a,\b}f = \sum_{n=0}^{\infty}
	e^{-t[ n(n-\a+\b+1) - \a (\b+1)]}
	\big\langle f, (1-x)^{-\a}\breve{P}_n^{-\a,\b}
		\big\rangle_{d\rho_{\ab}} \, (1-x)^{-\a}\breve{P}_n^{-\a,\b}, \qquad t \ge 0.
$$
The corresponding integral representation is
\begin{equation} \label{jsint2}
{\overset{\sim,\cdot}{T}}{}_{\! t}^{\a,\b}f(x) = \int_{-1}^1 {\overset{\sim,\cdot}{G}}{}_t^{\a,\b}(x,y)
	f(y)\, d\rho_{\ab}(y), \qquad x \in (-1,1), \quad t >0,
\end{equation}
with the integral kernel
\begin{align} \nonumber
{\overset{\sim,\cdot}{G}}{}_t^{\a,\b}(x,y) & = 
	\sum_{n=0}^{\infty}
	e^{-t[ n(n-\a+\b+1) - \a (\b+1)]}
	(1-x)^{-\a}\breve{P}_n^{-\a,\b}(x) \, (1-y)^{-\a}\breve{P}_n^{-\a,\b}(y) \\
& = e^{t \a (\b+1)} \big[(1-x)(1-y)\big]^{-\a} G_t^{-\a,\b}(x,y), \qquad x,y \in (-1,1), \quad t>0.
\label{p5}
\end{align}
We see that ${\overset{\sim,\cdot}{G}}{}_t^{\a,\b}(x,y)$ is strictly positive and smooth in
$(x,y,t) \in (-1,1)^2\times \R_+$. Using \eqref{p5} and Theorem \ref{jkerest}, it is straightforward to
check that for $\a < 0$ the integral \eqref{jsint2} converges absolutely for every
$f \in L^p(d\rho_{\ab})$, $1 \le p \le \infty$, providing a pointwise definition of
${\overset{\sim,\cdot}{T}}{}_{\! t}^{\a,\b}$ on all these $L^p$ spaces.
Moreover, for such $\a$ the family of 
operators $\{{\overset{\sim,\cdot}{T}}{}_{\! t}^{\a,\b}: t>0\}$
satisfies on the $L^p$ spaces the semigroup property,
as can be seen from \eqref{p5} and the analogous property for $\{T_t^{-\a,\b}:t>0\}$. On the other hand,
the case $0 < \a < 1$ is more subtle due to a pencil phenomenon.
The operators ${\overset{\sim,\cdot}{T}}{}_{\! t}^{\a,\b}$ are bounded on $L^p$ only for
$\a+1 < p < (\a+1)/\a$. In particular, ${\overset{\sim,\cdot}{T}}{}_{\! t}^{\a,\b}\mathds{1}$
is for each $t>0$ an unbounded function.

The situation when only $\b$ is exotic is completely parallel. Let $\a > -1$ and $0 \neq \b < 1$.
The relevant orthogonal system is $\{(1+x)^{-\b}P_n^{\a,-\b} : n \ge 0\}$ and we have
$$
J_{\ab} \big( (1+x)^{-\b} P_n^{\a,-\b} \big) = \big[ n(n+\a-\b+1) - (\a+1)\b\big] (1+x)^{-\b}P_n^{\a,-\b},
	\qquad n \ge 0.
$$
We denote the emerging self-adjoint extension by $J_{\ab}^{\textrm{cls,exo}}$ and the corresponding
exotic Jacobi semigroup by $\{{\overset{\cdot,\sim}{T}}{}_{\! t}^{\a,\b}\}$. Again, there is an integral
representation
\begin{equation} \label{jsint3}
{\overset{\cdot,\sim}{T}}{}_{\! t}^{\a,\b}f(x) = \int_{-1}^1 {\overset{\cdot,\sim}{G}}{}_t^{\a,\b}(x,y)
	f(y)\, d\rho_{\ab}(y), \qquad x \in (-1,1), \quad t >0,
\end{equation}
where
\begin{equation} \label{p55}
{\overset{\cdot,\sim}{G}}{}_t^{\a,\b}(x,y) = 
  e^{t (\a+1)\b} \big[(1+x)(1+y)\big]^{-\b} G_t^{\a,-\b}(x,y), \qquad x,y \in (-1,1), \quad t>0.
\end{equation}
Obviously, this kernel has properties parallel to ${\overset{\sim,\cdot}{G}}{}_t^{\a,\b}(x,y)$.
If $\b < 0$, \eqref{jsint3} defines pointwise a semigroup of operators on each $L^p(d\rho_{\ab})$,
$1 \le p \le \infty$. In the opposite case, when $0 < \b < 1$, the pencil phenomenon occurs, the condition
$\b+1 < p < (\b+1)/\b$ comes into play and in particular we see that  
${\overset{\cdot,\sim}{T}}{}_{\! t}^{\a,\b}\mathds{1}$ is an unbounded function for each $t>0$.

Finally, we focus on the situation when both $\a$ and $\b$ are exotic. In a sense, this is
a combination of the two previous exotic situations.
Assume that $0 \neq \a < 1$ and $0 \neq \b < 1$.
Consider the system $\{(1-x)^{-\a}(1+x)^{-\b} P_n^{-\a,-\b} : n \ge 0\}$. It is straightforward to verify
that this is an orthogonal basis in $L^2(d\rho_{\ab})$ and one has
$$
J_{\ab}\big( (1-x)^{-\a}(1+x)^{-\b} P_n^{-\a,-\b} \big) = \big[n(n-\a-\b+1)-\a-\b\big]
	(1-x)^{-\a}(1+x)^{-\b} P_n^{-\a,-\b}, \qquad n \ge 0.
$$
Further, 
$$
\frac{(1-x)^{-\a}(1+x)^{-\b} P_n^{-\a,-\b}}{\|(1-x)^{-\a}(1+x)^{-\b} P_n^{-\a,-\b}\|_{L^2(d\rho_{\ab})}}
= (1-x)^{-\a}(1+x)^{-\b} \breve{P}_n^{-\a,-\b}, \qquad n \ge 0.
$$ 
The self-adjoint extension of $J_{\ab}$ arising naturally in this context is given by
\begin{align*}
J_{\ab}^{\textrm{exo,exo}}f & = \sum_{n=0}^{\infty} \Big( \big[n(n-\a-\b+1)-\a-\b\big]
	\big\langle f, (1-x)^{-\a}(1+x)^{-\b} \breve{P}_n^{-\a,-\b}\big\rangle_{d\rho_{\ab}} \\
& \qquad \qquad \times		(1-x)^{-\a}(1+x)^{-\b} \breve{P}_n^{-\a,-\b} \Big)
\end{align*}
on the domain $\dom J_{\ab}^{\textrm{exo,exo}}$ consisting of all $f \in L^2(d\rho_{\ab})$ for which
the above series converges in $L^2(d\rho_{\ab})$. Notice that $J_{\ab}^{\textrm{exo,exo}}$ is non-negative
in the spectral sense if and only if $\a+\b \le 0$.

The exotic Jacobi semigroup ${\overset{\sim,\sim}{T}}{}_{\! t}^{\a,\b} = 
\exp(-t J_{\ab}^{\textrm{exo,exo}})$,
$t \ge 0$, defined spectrally in $L^2(d\rho_{\ab})$, has the integral representation
\begin{equation} \label{jsint4}
{\overset{\sim,\sim}{T}}{}_{\! t}^{\a,\b}f(x) = \int_{-1}^1 {\overset{\sim,\sim}{G}}{}_t^{\a,\b}(x,y)
	f(y)\, d\rho_{\ab}(y), \qquad x \in (-1,1), \quad t>0,
\end{equation}
where
\begin{equation} \label{p6}
{\overset{\sim,\sim}{G}}{}_t^{\a,\b}(x,y)  = e^{t(\a+\b)} \big[(1-x)(1-y)\big]^{-\a}
	\big[(1+x)(1+y)\big]^{-\b} G_t^{-\a,-\b}(x,y). 
\end{equation}
The kernel ${\overset{\sim,\sim}{G}}{}_t^{\a,\b}(x,y)$ is strictly positive and smooth in
$(x,y,t) \in (-1,1)^2\times \R_+$. By \eqref{p6} and Theorem \ref{jkerest} it follows that when
$\a < 0$ and $\b < 0$ the integral \eqref{jsint4} converges absolutely for $f \in L^p(d\rho_{\ab})$,
$1 \le p \le \infty$, providing a pointwise definition of ${\overset{\sim,\sim}{T}}{}_{\! t}^{\a,\b}$
on the $L^p$ spaces. Moreover, for such $\ab$ the family
$\{{\overset{\sim,\sim}{T}}{}_{\! t}^{\a,\b}:t>0\}$
satisfies the semigroup property on all these spaces. On the other hand, the opposite case 
$\a \vee \b > 0$ is more subtle because of the pencil phenomenon. Then 
${\overset{\sim,\sim}{T}}{}_{\! t}^{\a,\b}$ is bounded on $L^p$ only for
$\a \vee \b +1 < p < (\a \vee \b +1)/(\a \vee \b)$. In particular, 
${\overset{\sim,\sim}{T}}{}_{\! t}^{\a,\b}\mathds{1} \notin L^{\infty}$ for each $t>0$.
In contrast with previously considered situations, the pencil phenomenon occurs here also for
some parameters $\ab$ for which the operator $J_{\ab}^{\textrm{exo,exo}}$ is spectrally non-negative.

It is worth mentioning that there exist descriptions of the self-adjoint operators
$J_{\ab}^{\textrm{cls,cls}}$, $J_{\ab}^{\textrm{exo,cls}}$, $J_{\ab}^{\textrm{cls,exo}}$,
$J_{\ab}^{\textrm{exo,exo}}$ as differential operators, as in the Laguerre case; 
see Section \ref{sec:intro} and \cite{H2}. For instance,
\begin{equation} \label{zz}
\dom J_{\ab}^{\textrm{cls,cls}} = \big\{ f\in \mathcal{D}_{\ab} : 
	\lim_{x \to 1^-} (1-x)^{\a+1}f'(x) = 0, \,\lim_{x \to -1^+} (1+x)^{\b+1}f'(x) = 0 \big\},
\end{equation}
where $\mathcal{D}_{\ab}$ denotes the subspace of those $f \in L^2(d\rho_{\ab})$ for which $J_{\ab}f$
exists in a weak sense and is in $L^2(d\rho_{\ab})$. Then the action of $J_{\ab}^{\textrm{cls,cls}}$
on its domain is given by the differential operator $J_{\ab}$. Moreover, the two boundary conditions in
\eqref{zz} are automatically satisfied when $\ab \ge 1$.
For further details we refer to Hajmirzaahmad \cite{H1} and Hajmirzaahmad and Krall \cite{HK}.

Finally, for further reference we prove the following result.
\begin{thm} \label{thm:jacdom}
The following bounds hold uniformly in $x,y \in (-1,1)$ and $t>0$.
\begin{itemize}
\item[(a)]
If $-1 < \a < 0$ and $\b > -1$, then
${\overset{\sim,\cdot}{G}}{}_t^{\a,\b}(x,y) \le G_t^{\ab}(x,y)$.
\item[(b)]
If $\a > -1$ and $-1 < \b <0$, then
${\overset{\cdot,\sim}{G}}{}_t^{\a,\b}(x,y) \le G_t^{\ab}(x,y)$.
\item[(c)]
If $-1 < \a < 0$ and $-1 < \b < 0$, then
${\overset{\sim,\sim}{G}}{}_t^{\a,\b}(x,y) \le G_t^{\ab}(x,y)$.
\end{itemize}
\end{thm}

\begin{proof}
The comparison principle \cite[Theorem 3.5]{NoSj} says that for $x,y \in (-1,1)$ and $t>0$
$$
\big[(1-x)(1-y)\big]^{\epsilon/2} \big[(1+x)(1+y)\big]^{\delta/2} {G}_t^{\a+\epsilon,\b+\delta}(x,y)
	\le e^{\frac{\epsilon+\delta}2 (\a+\b+1+\frac{\epsilon+\delta}2)t} {G}_t^{\a,\b}(x,y)
$$
provided that $\ab > -1$ and $\epsilon,\delta \ge 0$ are such that
$\a \ge -\epsilon/2$ and $\b \ge -\delta/2$. 
The proof of this result in \cite{NoSj} shows that one can delete the assumption $\a \ge -\epsilon/2$
if $\epsilon = 0$, and similarly for the assumption $\b \ge -\delta/2$ in case $\delta = 0$.
Taking \eqref{p5} into account and applying the comparison
principle with $\epsilon = -2\a$ and $\delta = 0$ we get (a). Item (b) is analogous. To show (c), we use
\eqref{p6} and take $\epsilon = -2\a$ and $\delta = -2\b$ in the comparison principle.
\end{proof}

The three inequalities of Theorem \ref{thm:jacdom} are presumably strict.
To give a heuristic justification of this,
we invoke the probabilistic interpretation and point out
that the kernels involved are transition probability densities for Jacobi processes that differ only by
the nature of the boundary points. In the context of Theorem \ref{thm:jacdom},
the boundary points are either reflecting
or killing (more precisely, for the boundary point $x=1$ reflection corresponds to non-exotic
values of $\a$ and killing to exotic values of $\a$; similarly for $x=-1$ and $\b$).

An important conclusion from Theorem \ref{thm:jacdom} is that,
for the indicated ranges of the type parameters, the exotic Jacobi semigroups
$\{{\overset{\sim,\cdot}{T}}{}_{\! t}^{\a,\b}\}$, $\{{\overset{\cdot,\sim}{T}}{}_{\! t}^{\a,\b}\}$
and $\{{\overset{\sim,\sim}{T}}{}_{\! t}^{\a,\b}\}$
are contractive
on $L^{\infty}$ and thus on all $L^p(d\rho_{\ab})$, $1 \le p \le \infty$. The same should be true for
wider ranges of $\ab$. Actually, we believe that the following stronger result holds.
\begin{conj} \label{con:a}
Let $x \in (-1,1)$ and $t > 0$.
\begin{itemize}
\item[(a)]
If $\a < 0$ and $\b > -1$, then ${\overset{\sim,\cdot}{T}}{}_{\! t}^{\a,\b}\mathds{1}(x) < 1$.
\item[(b)]
If $\a > -1$ and $\b < 0$, then ${\overset{\cdot,\sim}{T}}{}_{\! t}^{\a,\b}\mathds{1}(x) < 1$.
\item[(c)]
If $\a < 0$ and $\b < 0$, then ${\overset{\sim,\sim}{T}}{}_{\! t}^{\a,\b}\mathds{1}(x) < 1$.
\end{itemize}
\end{conj}
We note that from Theorem \ref{jkerest} it follows that the bounds (a)--(c) of the above conjecture
hold with $1$ on their right-hand sides replaced by some constants depending only on $\a$ and $\b$.
This shows, in particular, that the semigroups involved are uniformly bounded on $L^{\infty}$ and
thus on all $L^p(d\rho_{\ab})$, $1 \le p \le \infty$.

\subsection{Multi-dimensional exotic Jacobi context and the maximal theorem}
We now pass to the multi-dimensional situation.
Let $d \ge 1$ and let $\ab \in \R^d$ be multi-parameters. Further, fix $\E,\F \subset \{1,\ldots,d\}$ and
assume that $\a \in A(\E)$ and $\b \in A(\F)$, see \eqref{def:ABC}. Define
\begin{align*}
\mathfrak{P}_n^{(\a,\E),(\b,\F)} & = \bigotimes_{i=1}^{d} 
\begin{cases}
(1-x_i)^{-\a_i}(1+x_i)^{-\b_i} \breve{P}_{n_i}^{-\a_i,-\b_i}, & \quad i \in \E \cap \F, \\
(1-x_i)^{-\a_i}\breve{P}_{n_i}^{-\a_i,\b_i}, & \quad i \in \E \setminus \F, \\
(1+x_i)^{-\b_i}\breve{P}_{n_i}^{\a_i,-\b_i}, & \quad i \in \F \setminus \E, \\
\breve{P}_{n_i}^{\a_i,\b_i}, & \quad i \notin \E \cup \F,
\end{cases}
\qquad n \in \N^d.
\end{align*}
Then the system $\{\mathfrak{P}_n^{(\a,\E),(\b,\F)} : n \in \N^d\}$ is an orthonormal basis
in $L^2(d\rho_{\ab})$, where now $\rho_{\ab}$ is the product measure in $(-1,1)^d$ given by
$$
d\rho_{\ab}(x) = (\mathds{1}-x)^{\a} (\mathds{1}+x)^{\b}\, dx.
$$
Furthermore, $\mathfrak{P}_n^{(\a,\E),(\b,\F)}$ are eigenfunctions of the Jacobi differential
operator $\mathbb{J}_{\ab} = \sum_{i=1}^d J_{\a_i,\b_i}$ (each component of the sum acting on the
corresponding coordinate variable) and one has $\mathbb{J}_{\ab} \mathfrak{P}_n^{(\a,\E),(\b,\F)}
= \lambda_n^{(\a,\E),(\b,\F)} \mathfrak{P}_n^{(\a,\E),(\b,\F)}$, $n \in \N^d$,
where
\begin{align*}
\lambda_n^{(\a,\E),(\b,\F)} & = 
\sum_{i \in \E \cap \F} \big[ n_i(n_i-\a_i-\b_i+1) -\a_i-\b_i\big] + 
	\sum_{i \in \E \setminus \F} \big[ n_i(n_i-\a_i+\b_i+1) -\a_i(\b_i+1)\big] \\
& \qquad + \sum_{i \in \F \setminus \E} \big[ n_i(n_i+\a_i-\b_i+1) -(\a_i+1)\b_i\big]
 + \sum_{i \in \E^c \cap \F^c} n_i(n_i+\a_i+\b_i+1).
\end{align*} 

We consider the self-adjoint extension of $\mathbb{J}_{\ab}$
(acting initially on $\spann\{\mathfrak{P}_n^{(\a,\E),(\b,\F)} : n \in \N^d\} \subset L^2(d\rho_{\ab})$)
defined by
$$
\mathbb{J}_{(\a,\E),(\b,\F)}f = \sum_{n \in \N^d} \lambda_n^{(\a,\E),(\b,\F)}
	\big\langle f, \mathfrak{P}_n^{(\a,\E),(\b,\F)} \big\rangle_{d\rho_{\ab}} \,
	\mathfrak{P}_n^{(\a,\E),(\b,\F)}
$$
on the domain $\dom \mathbb{J}_{(\a,\E),(\b,\F)}$ consisting of all $f \in L^2(d\rho_{\ab})$
for which the above series converges in $L^2(d\rho_{\ab})$. Notice that the classical multi-dimensional
Jacobi setting is naturally embedded in this general setup and it corresponds to
$\E=\F=\emptyset$. Otherwise, if $\E \cup \F \neq \emptyset$,
we are concerned with the exotic situation.

The semigroup $\mathbb{T}_t^{(\a,\E),(\b,\F)} = \exp(-t\mathbb{J}_{(\a,\E),(\b,\F)})$, $t \ge 0$,
defined spectrally in $L^2(d\rho_{\ab})$, has the integral representation
\begin{equation} \label{eJir}
\mathbb{T}_t^{(\a,\E),(\b,\F)}f(x) = \int_{(-1,1)^d} \mathbb{G}_t^{(\a,\E),(\b,\F)}(x,y)
	f(y)\, d\rho_{\ab}(y), \qquad x \in (-1,1)^d, \quad t>0,
\end{equation}
where, for $x,y \in (-1,1)^d$ and $t>0$,
\begin{align*}
& \mathbb{G}_t^{(\a,\E),(\b,\F)}(x,y) \\ &\quad   = 
	\prod_{i \in \E \cap \F} {\overset{\sim,\sim}{G}}{}_t^{\a_i,\b_i}(x_i,y_i)
	\prod_{i \in \E \setminus \F} {\overset{\sim,\cdot}{G}}{}_t^{\a_i,\b_i}(x_i,y_i) 
	\prod_{i \in \F \setminus \E} {\overset{\cdot,\sim}{G}}{}_t^{\a_i,\b_i}(x_i,y_i)
	\prod_{i \in \E^c \cap \F^c} G_t^{\a_i,\b_i}(x_i,y_i).
\end{align*}
Obviously, $\mathbb{G}_t^{(\a,\E),(\b,\F)}(x,y)$ is strictly positive, symmetric and smooth
in $(x,y,t) \in (-1,1)^{2d} \times \R_+$ (this and some other facts below are immediate
consequences of the analogous properties of the one-dimensional components).

When $\mE(\a) < 0$ and $\mF(\b) < 0$, the formula \eqref{eJir} provides a pointwise definition
of $\mathbb{T}_t^{(\a,\E),(\b,\F)}f$ on all $L^p(d\rho_{\ab})$ spaces, $1 \le p \le \infty$.
Moreover, $\{\mathbb{T}_t^{(\a,\E),(\b,\F)} : t>0\}$
is a uniformly bounded semigroup of operators on each $L^p(d\rho_{\ab})$, $1 \le p \le \infty$. We strongly 
believe that this is in fact a semigroup of contractions on each $L^p(d\rho_{\ab})$,
but at the moment we are not able to justify this strictly in full generality (that is, for all considered
$\a$ and $\b$), see the related discussion in the one-dimensional case and Conjecture \ref{con:a}. Thus
for general $\ab$ we conjecture that $\{\mathbb{T}_t^{(\a,\E),(\b,\F)}\}$ is a submarkovian
symmetric diffusion semigroup, which is Markovian if and only if $\E = \F = \emptyset$.

In case $\mE(\a) > 0$ or $\mF(\b) > 0$ a pencil phenomenon
occurs. Given $t > 0$, the operator $\mathbb{T}_t^{(\a,\E),(\b,\F)}$ is well defined in
$L^p(d\rho_{\ab})$ only if
\begin{equation} \label{rr}
p > 1 + \mE(\a) \vee \mF(\b);
\end{equation}
in particular, $\mathbb{T}_t^{(\a,\E),(\b,\F)}$ is not defined in $L^1(d\rho_{\ab})$.
Moreover, the requirement that $\mathbb{T}_t^{(\a,\E),(\b,\F)}$ maps $L^p(d\rho_{\ab})$ into 
$L^p(d\rho_{\ab})$ forces a dual restriction to \eqref{rr}. In particular, 
$\mathbb{T}_t^{(\a,\E),(\b,\F)}\mathds{1}$ is an unbounded function.

The principal objective of this section is to investigate the maximal operator
$$
\mathbb{T}_*^{(\a,\E),(\b,\F)}f = \sup_{t>0} \big| \mathbb{T}_t^{(\a,\E),(\b,\F)}f \big|.
$$
We assume that no pencil phenomenon occurs, which limits our considerations to
$\ab$ satisfying $\mE(\a) < 0$ and $\mF(\b) < 0$.
We will prove that the maximal operator satisfies the weak type
$(1,1)$ estimate. Since $\mathbb{T}_*^{(\a,\E),(\b,\F)}$ is bounded on $L^{\infty}$
(see the comment following Conjecture \ref{con:a}),
by interpolation this will also give $L^p(d\rho_{\ab})$ boundedness, $1 < p < \infty$.
The latter could also be concluded from Stein's maximal theorem \cite[p.\,73]{St}, but only
for those $\ab$ for which we know so far that the semigroup is $L^p$-contractive.

The main result of this section is the following.
\begin{thm} \label{thm:weakJd}
Let $d \ge 1$ and $(\ab) \in A(\E) \times A(\F)$ for some $\E,\F \subset \{1,\ldots,d\}$.
Assume that $\mE(\a) < 0$ and $\mF(\b) < 0$.
Then $\mathbb{T}_*^{(\a,\E),(\b,\F)}$ is bounded from $L^1(d\rho_{\ab})$ to weak $L^1(d\rho_{\ab})$.
\end{thm}

Notice that in the special case $\E=\F=\emptyset$, Theorem \ref{thm:weakJd} says that the
classical (non-exotic) Jacobi semigroup maximal operator is of weak type $(1,1)$. This result is
already known, see \cite[Theorem 5.1]{NoSj}. Actually, it is a straightforward consequence of the
Gaussian upper bound for the classical $d$-dimensional Jacobi kernel $G_t^{\ab}(x,y)$
in the trigonometric parameterization $x_i = \cos\theta_i$, $y_i = \cos\varphi_i$, $i=1,\ldots,d$,
that in dimension one was obtained essentially in \cite[Theorem 7.2]{CKP}.
The Gaussian bound is readily seen
from Theorem \ref{jkerest} and the relation (see \cite[Lemma 4.2]{NoSj7})
$$
\eta_{\ab}\big( B(\theta,r) \big) \simeq 1 \wedge \Big( r^d \prod_{i=1}^d (\theta_i+r)^{2\a_i+1}
	(\pi-\theta_i+r)^{2\b_i+1} \Big), \qquad \theta \in (0,\pi)^d, \quad r > 0,
$$
where $\ab > -1$ and $\eta_{\ab}$ is a doubling measure defined in \eqref{etam} below.
We leave further details to interested readers.

When proving Theorem \ref{thm:weakJd}, we may restrict to $f \ge 0$, since the kernel
$\mathbb{G}_t^{(\a,\E),(\b,\F)}(x,y)$ is positive.
Moreover, using Theorem \ref{jkerest} we infer that 
$$
\mathbb{G}_t^{(\a,\E),(\b,\F)}(x,y) \lesssim \mathbb{G}_1^{(\a,\E),(\b,\F)}(x,y),
	\qquad x,y \in (-1,1)^d, \quad t \ge 1,
$$ 
and therefore we may assume that the supremum in the definition of $\mathbb{T}_*^{(\a,\E),(\b,\F)}$
is taken only over $0 < t \le 1$.
Furthermore, it is convenient to apply the trigonometric parameterization
$$
x = \cos\theta := (\cos\theta_1,\ldots,\cos\theta_d), \qquad
y = \cos\varphi := (\cos\varphi_1,\ldots,\cos\varphi_d), \qquad \theta,\varphi \in (0,\pi)^d.
$$ 
Instead of $\rho_{\ab}$, we use the equivalent measure
\begin{equation} \label{etam}
d\eta_{\ab}(\theta) = 
\prod_{i=1}^d \theta_i^{2\a_i+1}(\pi-\theta_i)^{2\b_i+1}\, d\theta, \qquad \t \in (0,\pi)^d.
\end{equation}
Changing the variables, we see that Theorem \ref{thm:weakJd} will follow once we show that
the weak type bound
\begin{align*}
& \eta_{\ab}\Big\{ \theta \in (0,\pi)^d : \sup_{0 < t \le 1} \int_{(0,\pi)^d}
	\mathbb{G}_t^{(\a,\E),(\b,\F)}(\cos\theta,\cos\varphi) f(\varphi)\, d\eta_{\ab}(\varphi) > \lambda
	\Big\} \\
& \qquad  \le \frac{C}{\lambda} \int_{(0,\pi)^d} f(\theta)\, d\eta_{\ab}(\theta)
\end{align*}
holds with a constant $C$ independent of $\lambda > 0$ and $f \ge 0$.

To proceed, we take into account Theorem \ref{jkerest} and the relations \eqref{p5}, \eqref{p55}, \eqref{p6}
and introduce several auxiliary one-dimensional kernels that will be used to control
$\mathbb{G}_t^{(\a,\E),(\b,\F)}(\cos\theta,\cos\varphi)$:
\begin{align*}
K_{t}^{\ab} (\t,\vp) & = 
\big[ \t \vp + t\big]^{-\a-1/2}
	\big[ (\pi-\theta)(\pi-\varphi) + t]^{-\b-1/2}
	\frac{1}{\sqrt{t}} \exp\bigg( - c_2\frac{(\theta-\varphi)^2}{t}\bigg),
\qquad \!\!\!\ab > -1, \\
{\overset{\sim,\cdot}{K}}{}_{t}^{\ab} (\t,\vp) & = (\t \vp)^{-2\a} K_{t}^{-\a,\b} (\t,\vp),
\qquad \a < 0, \quad \b > -1, \\
{\overset{\cdot,\sim}{K}}{}_{t}^{\ab} (\t,\vp) & 
= \big[ (\pi-\theta)(\pi-\varphi)]^{-2\b} K_{t}^{\a,-\b} (\t,\vp), \qquad \a > -1, \quad \b < 0, \\
{\overset{\sim,\sim}{K}}{}_{t}^{\ab} (\t,\vp) & 
= (\t \vp)^{-2\a} \big[ (\pi-\theta)(\pi-\varphi)]^{-2\b} K_{t}^{-\a,-\b} (\t,\vp), \qquad \ab < 0;
\end{align*}
here $\theta,\varphi \in (0,\pi)$, $0 < t \le 1$
and $c_2>0$ is the constant from Theorem \ref{jkerest}.
Given $(\ab) \in A(\E) \times A(\F)$ such that $\mE(\a), \mF(\b) < 0$, we now define the
multi-dimensional kernel
\begin{align*}
& \mathbb{K}_{t}^{(\a,\E),(\b,\F)} (\t,\vp) \\
& \quad = \prod_{i \in \E \cap \F} {\overset{\sim,\sim}{K}}{}_{t}^{\a_i,\b_i}(\t_i,\vp_i) 
	\prod_{i \in \E \setminus \F} {\overset{\sim,\cdot}{K}}{}_{t}^{\a_i,\b_i}(\t_i,\vp_i)
\prod_{i \in \F \setminus \E} {\overset{\cdot,\sim}{K}}{}_{t}^{\a_i,\b_i}(\t_i,\vp_i)
\prod_{i \in \E^c \cap \F^c} K_{t}^{\a_i,\b_i}(\t_i,\vp_i),
\end{align*}
where $\theta,\varphi \in (0,\pi)^d$ and $0 < t \le 1$, and introduce the associated maximal operator
$$
\mathbb{K}_{*}^{(\a,\E),(\b,\F)} f(\t) = \sup_{0<t \le 1} \int_{(0,\pi)^d} 
\mathbb{K}_{t}^{(\a,\E),(\b,\F)} (\t,\vp)  f(\vp)\, d\eta_{\ab}(\vp),
\qquad \t \in (0,\pi)^d 
$$
acting on, say, $0 \le f \in L^1(d\eta_{\ab})$. In view of Theorem \ref{jkerest},
\[
\mathbb{G}_t^{(\a,\E),(\b,\F)} (\cos\theta,\cos\varphi)
\lesssim
\mathbb{K}_{t}^{(\a,\E),(\b,\F)} (\t,\vp),
\qquad \t,\vp \in (0,\pi)^d, \quad 0< t \le 1.
\]
Therefore, proving Theorem \ref{thm:weakJd} reduces now to showing the following result.

\begin{thm} \label{thm:weakJr}
Let $d \ge 1$ and $(\ab) \in A(\E) \times A(\F)$ for some $\E,\F \subset \{1,\ldots,d\}$.
Assume that $\mE(\a), \mF(\b) < 0$. Then the bound
$$
\eta_{\ab} \big\{\t \in (0,\pi)^d: \mathbb{K}_{*}^{(\a,\E),(\b,\F)} f(\t)> \lambda \big\} 
	\le \frac{C}{\lambda} \int_{(0,\pi)^d}f(\t)\, d\eta_{\ab}(\t)
$$
holds with a constant $C$ independent of $\lambda > 0$ and $f \ge 0$.
\end{thm}

The proof of Theorem \ref{thm:weakJr} boils down in a straightforward manner to the two lemmas below.
To state them we need some additional notation. 
For a given $\W \subset \{1,\ldots,d\}$ denote
\[
Q(\W) =  
\prod_{i=1}^d
\begin{cases}	
	(\pi/4,\pi), & \quad i \in \W, \\
	(0,3\pi/4), & \quad i \notin \W.
\end{cases}
\]

\begin{lem}\label{lem:onJac}
Let $d \ge 1$, $\E,\F,\W \subset \{1,\ldots,d\}$ and $(\ab) \in A(\E) \times A(\F)$.
Assume that $\mE(\a), \mF(\b) < 0$.
Then the maximal operator
$$
\mathbb{L}_{*,\W}^{(\a,\E),(\b,\F)}f(\t) = \sup_{0 < t \le 1} \int_{(0,\pi)^d} 
	\chi_{\{\t,\vp \in Q(\W)\}} 
	 \mathbb{K}_{t}^{(\a,\E),(\b,\F)} (\t,\vp)  f(\vp)\, d\eta_{\ab}(\vp),
 \qquad f \ge 0,
$$
is of weak type $(1,1)$ with respect to $\eta_{\ab}$.
\end{lem}

\begin{lem}\label{lem:offJac}
In dimension $d=1$, let $\E,\F \subset \{1\}$ and $(\ab) \in A(\E) \times A(\F)$.
Assume that $\mE(\a), \mF(\b) < 0$. Then the operator
$$
N^{(\a,\E),(\b,\F)} f(\t) = \int_0^\pi 
	\chi_{\{ |\t - \vp| \ge \pi/2 \}} 
	 \bigg[\sup_{0 < t \le 1} \mathbb{K}_{t}^{(\a,\E),(\b,\F)}(\t,\vp)\bigg]  f(\vp)\, d\eta_{\ab}(\vp),
 \qquad f \ge 0,
$$
is of strong type $(1,1)$ with respect to $\eta_{\ab}$.
\end{lem}

It remains to prove these two lemmas. The idea of the first proof below relies on reducing
the problem to an application of the maximal theorem from the Bessel context that was obtained in
Section \ref{sec:Bes}.

\begin{proof}[{Proof of Lemma \ref{lem:onJac}}] 
We first consider the special case $\W = \emptyset$, which means the restriction $\t,\vp \in (0,3\pi/4)^d$.
Observe that in the one-dimensional case, for $\gamma, \delta > -1$ fixed, we have 
\begin{align*} 
K_{t}^{\gamma, \delta}(\t,\vp) 
& \simeq
\big[ \t \vp + t\big]^{-\gamma -1/2}
	\frac{1}{\sqrt{t}} \exp\bigg( - c_2\frac{(\theta-\varphi)^2}{t}\bigg) \\
& \simeq 
W_{t/(4c_2)}^{\gamma} (\t,\vp),
\qquad 
\t,\vp \in (0,3\pi/4), \quad 0 < t \le 1, 
\end{align*} 
where $W_t^{\gamma}$ is the kernel of the classical one-dimensional Bessel semigroup considered in
Section~\ref{sec:Bes}; the last comparability above follows from \eqref{Bclcomp}.
Consequently, if $\mE(\a), \mF(\b) < 0$ we obtain
\[
\mathbb{K}_{t}^{(\a,\E),(\b,\F)} (\t,\vp) 
\simeq 
\mathbb{W}_{t/(4c_2)}^{\a,\E} (\t,\vp), 
\qquad 
\t,\vp \in (0,3\pi/4)^d, \quad 0 < t \le 1,
\]
which together with the fact that the measures $\eta_{\ab}$ and $\eta_{\a}$ are comparable in $(0,3\pi/4)^d$
implies
\[
\mathbb{L}_{*,\emptyset}^{(\a,\E),(\b,\F)}f(\t)
\lesssim
\chi_{(0,3\pi/4)^d } (\t) \,
\mathbb{W}_{*}^{\a,\E} \big( f  \chi_{(0,3\pi/4)^d } \big) (\t), \qquad \theta \in (0,\pi)^d.
\]
The conclusion for $\W = \emptyset$ now follows directly from Theorem \ref{thm:weakB}.

The general case $\W \subset \{1,\ldots,d\}$ follows from this
by simple symmetry arguments. To proceed, define sets
\begin{align*}
\check{\E} = (\F \cap \W) \cup (\E \cap \W^c),
\qquad
\check{\F} = (\E \cap \W) \cup (\F \cap \W^c),
\end{align*}
and multi-parameters $\check{\a},\check{\b}$ by
\begin{align*}
\check{\a}_i = 
\begin{cases}		
			\b_i, & i \in \W, \\
			\a_i, & i \in \W^c,
	\end{cases}
\qquad
\check{\b}_i = 
\begin{cases}		
			\a_i, & i \in \W, \\
			\b_i, & i \in \W^c.
	\end{cases}
\end{align*}
Further, let $\Psi \colon (0,\pi)^d \to (0,\pi)^d$ be the bijection determined by
\begin{align*}
\Psi (\t)_i = 
\begin{cases}
		\pi - \t_i, & i \in \W,\\
		\t_i, & i \notin \W,			
	\end{cases} \qquad i =1,\ldots,d.
\end{align*} 
Then one easily checks that 
\[
\| f\circ \Psi \|_{ L^1(d\eta_{\check{\a}, \check{\b}}) }
=
\| f \|_{L^1(d\eta_{\ab})},
\qquad
\big(\mathbb{L}_{*,\W}^{(\a,\E),(\b,\F)}f\big)\circ\Psi
=
\mathbb{L}_{*,\emptyset}^{(\check{\a},\check{\E}),(\check{\b},\check{\F})}(f \circ\Psi)
\]
(notice that $m_{\check{\E}}(\check{\a}), m_{\check{\F}}(\check{\b}) < 0$ if and only if
$\mE(\a), \mF(\b) < 0$).
Combining these relations with the already known weak type $(1,1)$ of the operator
$\mathbb{L}_{*,\emptyset}^{(\check{\a},\check{\E}),(\check{\b},\check{\F})}$,
we arrive at the desired conclusion.
\end{proof}

\begin{proof}[{Proof of Lemma \ref{lem:offJac}}]
To begin with, we claim that the kernel of $N^{(\a,\E),(\b,\F)}$ is controlled
by $\mathbb{K}_{1}^{(\a,\E),(\b,\F)} (\t,\vp)$.
To see this it is enough to check that, for $\ab > -1$ fixed, the bound
\begin{align*} 
\chi_{\{ |\t - \vp| \ge \pi/2 \}}
\big[ \t \vp + t\big]^{-\a-1/2}
	\big[ (\pi-\theta)(\pi-\varphi) + t]^{-\b-1/2}
	\frac{1}{\sqrt{t}} \exp\bigg( - c_2\frac{(\theta-\varphi)^2}{t}\bigg)
\lesssim 1,
\end{align*} 
holds uniformly in $\t,\vp \in (0,\pi)$ and $0 < t \le 1$. This, however, is straightforward.
Indeed, the left-hand side here is easily dominated, up to a multiplicative constant, by
${t^{-\a-\b-5/2}} \exp(-c_2 \frac{\pi^2}{4t}) \lesssim 1$. The claim follows.

Now, to conclude it suffices to verify that
\[
\int_0^\pi \mathbb{K}_{1}^{(\a,\E),(\b,\F)} (\t,\vp) \, d\eta_{\ab} (\t) 
\lesssim 1,
\qquad \vp \in (0,\pi),
\] 
which is elementary. Details are left to the reader.
\end{proof}

Now Theorem \ref{thm:weakJr} and thus also Theorem \ref{thm:weakJd} are proved.


\end{document}